\documentclass[11pt]{amsart}
\usepackage{amsmath,amsthm}
\usepackage {latexsym}
\usepackage{amssymb}
\usepackage{xcolor}

\usepackage[utf8]{inputenc}
\newcommand{\beal}{\begin{align}}
\newcommand{\enal}{\end{align}}
\newcommand{\bealn}{\begin{align*}}
\newcommand{\enaln}{\end{align*}}
\newcommand{\bear}{\begin{eqnarray}}
\newcommand{\eear}{\end{eqnarray}}
\newcommand{\beeq}{\begin{equation}}
\newcommand{\eneq}{\end{equation}}

\newcommand{\supp}{\mbox{\rm supp}}

\newcommand{\eps}{{\varepsilon}}
\newcommand{\R}{{\mathbb R}}

\newcommand{\sign}{\mbox{sign}}

\def\bm{\left[ \begin{array}{cc}}
\def\endm{\end{array}\right]}

\def\sign{{\rm sign}}

\def\eps{\varepsilon}

\def\bm{\left[\begin{matrix} }
\def\endm{\end{matrix}\right]}

\def\R{{\mathbb R}}
\def\S{\mathbb{S}^1}

\newtheorem{theorem}{Theorem}

\newtheorem{lemma}[theorem]{Lemma}

\newtheorem{cor}[theorem]{Corollary}
\newtheorem{prop}[theorem]{Proposition}

\newtheorem{conj}[theorem]{Conjecture}

\theoremstyle{remark}

\renewcommand{\hat}{\widehat}
\renewcommand{\epsilon}{\eps}
\renewcommand{\tilde}{\widetilde}
\numberwithin{equation}{section}
\numberwithin{theorem}{section}

\date{\today}

\begin{document}
\title{Semi-global controllability of a geometric wave equation}

\author{Joachim Krieger}
\address{Bâtiment des Mathématiques, EPFL\\Station 8, CH-1015 Lausanne, Switzerland}
\email{\texttt{joachim.krieger@epfl.ch}}
\thanks{}

\author{Shengquan Xiang}
\address{Bâtiment des Mathématiques, EPFL\\
 Station 8, CH-1015 Lausanne, Switzerland}
\email{\texttt{shengquan.xiang@epfl.ch.}}
\thanks{}

\begin{abstract}
We prove the semi-global controllability and stabilization  of the  $(1+1)$-dimensional  wave maps equation with spatial domain $\mathbb{S}^1$ and target $\mathbb{S}^k$. First we show that damping stabilizes the system when the energy is strictly below the threshold $2\pi$, where harmonic maps appear as obstruction for global stabilization. Then, we adapt an iterative control procedure to get low-energy exact controllability of the wave maps equation. This result is optimal in the case $k=1$.
\end{abstract}
\maketitle

\section{Introduction}
\subsection{The wave maps}

Wave maps are prototypes of geometric wave equations. Let us be given a Riemannian manifold $(\mathcal{M}, g)$ and the classical $(\mathbb{R}^{1+ n}, h)$ with Minkowski metric $h= \textrm{diag} (-1, 1,..., 1)$. If we consider functions with geometric target, \[\phi:\mathbb{R}\times \mathbb{R}^{n}\rightarrow  \mathcal{M}\] with Lagrangian of the form 
\[  L^h_{\mathcal{M}}= \int_{\mathbb{R}^{1+ n}} -|\partial_t \phi|_g^2+ |\nabla_x \phi|_g^2\, dt dx,\]
the Euler-Lagrange equation is given by 
\[ D^{\alpha} \partial_{\alpha} \phi=0,\]
which, in local coordinates, has the form 
\[ \Box \phi^i+ \Gamma^i_{jk}(\phi)\partial^{\alpha}\phi^j \partial_{\alpha} \phi^k= 0,\]
with d'Alembertian  $\Box$ referring to $-\partial_{tt}+ \Delta$.
We shall also keep in mind the special case with $\mathcal{M}$ being a submanifold of $\mathbb{R}^m$ equipped with the Euclidean metric, such as $\mathbb{S}^d\subset \mathbb{R}^{d+1}$. In these cases, thanks to the special structure of the second fundamental form,  the wave maps equation becomes
\[ \Box\phi+ S_{\phi}(\partial^{\alpha}\phi, \partial_{\alpha}\phi)= 0,  \]
where $S_p$ is a symmetric quadratic form on the tangent space mapping into the normal space. \\

The well-posedness and singularity formation issues of wave maps have been extensively studied in the past few decades. In general the well-posedness in $H^s(\mathbb{R}^n)$-spaces is related to various factors, such as the energy critical dimension $n=2$, 
the Sobolev degree threshold $s= \frac{n}{2}$, the scale of data,  the lifespan, and  geometric features of the target $\mathcal{M}$.  There is a huge amount of literature dedicated  to this topic, see for example the works by Christodoulou--Tahvildar-Zadeh \cite{Christodoulou-Tahvildar-Zadeh-1993-Duke, Christodoulou-Tahvildar-Zadeh-1993}, Klainerman-Machedon \cite{Klainerman-Machedon-1997}, Tao \cite{Tao-wwm-1, Tao-wwm-2, tao2009global3, tao2009global4, tao2009global5, tao2009global6, tao2009global7},  Tataru \cite{Tataru-2001}, Sterbenz-Tataru \cite{Ster-Tat1, Ster-Tat2}, Krieger-Schlag \cite{Krieger-Schlag-2012}. We also refer to  the survey by Tataru \cite{Tataru-2004} and the references therein.  The evolution of wave maps  becomes more delicate for large data in energy critical (and supercritical) cases, as finite time blow up may occur. Singularity formation has been among one of the central topics of dispersive equations in the last few decades, see for example \cite{Krieger-Schlag-Tataru,   R-R-2012, Rod-Ster} in the wave maps context.

Remark that in some situations wave maps admit non trivial stationary states:  solutions of  the elliptic equation,
\[ \Delta \phi^i+ \Gamma^i_{jk}(\phi)\partial^{\alpha}\phi^j \partial_{\alpha} \phi^k= 0.\]
These solutions are the so-called \textit{harmonic maps}. Harmonic maps form a central mathematical topic with connections to various branches of mathematical physics, we refer to the lecture notes by Schoen and Yau \cite{Schoen-Yau-harmonicmap} for an introduction to them.  It is natural that these objects are crucial to the study of wave maps, for instance in \cite{Krieger-Schlag-Tataru}  blow up solutions are explicitly found as approximately self-shrinking harmonic maps. In this paper we  further observe that harmonic maps appear as obstruction for global damping stabilization. \\

In this work, more specifically, we focus on   the $(1+1)$--dimensional wave maps with sphere as target  $\phi: \R\times \mathbb{S}^1\rightarrow \mathbb{S}^k$. The quadratic form $S_p$ is further simplified and the wave maps equation becomes
\begin{equation*}
  \Box \phi=   \left(|\phi_t|^2- |\phi_x|^2\right)\phi, \; \phi[0]= (f, g),
\end{equation*}
where $\phi[t]$ denotes $(\phi, \phi_t)(t)$.
For ease of presentation we also lift them to  $2\pi$-periodic wave maps $\phi(t, x): \R\times \R\rightarrow \mathbb{S}^k\subset \R^{k+ 1}$ while keeping the same notation. 
This system is globally well-posed for large data in $H^{s}$ with $s>3/4$ according to  \cite{Keel-Tao-1998}.
Finally we shall notice  that wave maps conserve the energy,
\begin{equation*}
    E(t):= \int_{\mathbb{S}^1} \left(|\phi_t|^2+ |\phi_x|^2\right)(t, x) dx= constant.
\end{equation*}

\subsection{The controlled wave maps}
Control theory aims at controlling systems using additional control terms. For example, a general linear control system in finite dimension or infinite dimension can be written as 
\begin{equation*}
 \dot{x}= Ax+ Bu   
\end{equation*}
where $x$ is the state and $u$ is the control. 
Typical problems in control theory are controllability problems and stabilization problems.  We refer to the book by Coron \cite{coron} for an excellent introduction on this topic.  \\
\textit{Exact controllability} in a set $H$ means that  by choosing suitable controls we are able to steer any given initial state to  another desired final state in $H$. In general $H$ is chosen as a Hilbert space or a small ball inside some Hilbert space. In the latter case we shall call it \textit{local controllability}, since states are supposed to be small (thus ``local")\footnote{Remark that in control theory usually the terminology ``local" is used to describe the scale of states, namely small data, which is different from some other branches of PDEs' study, for example ``local well-posedness" usually refers to well-posedness in a small time period. }.  In this paper we only deal with exact controllability problems. Other weaker controllability properties are also natural to investigate, namely \textit{null controllability}, \textit{approximate controllability} and \textit{controllability around trajectory}. \\
\textit{Stabilization} is an action to stabilize systems with the help of  suitable feedback controls, $i. e. $  controls are governed by feedback laws and we are interested in the stability of closed-loop systems.  Can we make the system asymptotically stable,  exponentially stable or even rapidly stable (namely, the exponential decay rate can be as large as we want)?\\

In this paper we  are interested in control problems related to wave maps. 
For any initial state $u[0]$ taking values in  $\mathbb{S}^k\times T\mathbb{S}^k$ and any $\mathbb{R}^{k+1}$-valued source term $f(t, x)$ having regularity   $L^2(0, T; L^2_x(\S))$   we can look at  the following inhomogeneous system:
\begin{equation}\label{eq:inhomowavemaps}
     \Box \phi= \left(|\phi_{t}|^2- |\phi_{ x}|^2\right) \phi+ 
     f^{\phi^{\perp}}, \; \phi[0]= u[0]
\end{equation}
where, by $f^{\phi^{\perp}}$ we refer to the orthogonal projection of $f$ onto the tangent space $\phi^{\perp}$.
In particular,  when $f(t, x)$ is chosen to coincide with ${\bf 1}_{\omega} f$ we call  the preceding system  an {\it{internally controlled wave maps equation}}, since controls are  supported in an internal domain $\omega\subset \mathbb{S}^1$. Moreover, the precise control term is given \textit{implicitly} by ${\bf 1}_{\omega} f^{\phi^{\perp}}$.\\

First remark that the inhomogeneous wave maps equation \eqref{eq:inhomowavemaps} is  $\mathbb{S}^k$ invariant, namely $\phi[t]:  \mathbb{S}^1\rightarrow \mathbb{S}^k\times T\mathbb{S}^k$. Actually, if  we denote  by $y(t, x)$ the value of $|\phi(t, x)|^2$, then the conditions on the initial state $\phi[0]$ become  $$y(0, x)=1 \textrm{ and } y_t(0, x)=0, \; \forall x\in \mathbb{S}^1,$$
while the wave maps equation turns into a scalar wave equation in terms of $y$:
\begin{align*}
   \Box y= 2\langle \Box \phi, \phi\rangle- 2\left(|\phi_t|^2-|\phi_x|^2\right)
   = 2\left(|\phi_t|^2-|\phi_x|^2\right)(y- 1).
\end{align*}
Since the preceding wave equation admits a unique solution $(y, y_t)= (1, 0)$,
 the controlled wave maps equation stays on $\mathbb{S}^k$:
\begin{equation*}
    |\phi(t, x)|=1, \; \forall t\in \mathbb{R}, \forall x\in \mathbb{S}^1.
\end{equation*}

In this geometric framework by performing ``linearization" around trivial equilibrium points we ``roughly" arrive at the controlled wave equation, where we have also ignored  
 the geometric constraints upon controls and states,
 \begin{equation*}
     \Box \phi= {\bf 1}_{\omega} f.
 \end{equation*}
 The controllability of wave equations has been heavily investigated in the past decades based on the Hilbert Uniqueness Method (HUM) introduced by Lions \cite{Lions}. Heuristically speaking, HUM transforms controllability problems into (quantitative) unique continuation problems, which are also known as duality and observability inequalities.  Based on this observation satisfactory  controllability results have been discovered using  different analytic techniques, for example, in cases that the controlled domain verifies certain conditions, the  multiplier method leads to explicit observability inequalities \cite{Lagnese}; on the other hand microlocal analysis provides an (almost) necessary and sufficient condition for controllability properties, the so called Geometric Control Condition (GCC) introduced by Bardos--Lebeau--Rauch \cite{Bardos-Lebeau-Rauch} (we also refer to \cite{Burq-Gerard-wave} for a simplified proof),  for which, however, the control is not always explicit due to the use of compactness arguments. 
 
 In particular the preceding wave equation is exactly controllable in the space $H^1_x\times L^2_x(\S)$, see Lemma~\ref{lem:conlwm} for a detailed statement of this result. This motivates the first  question that we may ask for the controlled wave maps equation.
 
\noindent {\bf Question 1:}
    \textit{Is the controlled wave maps equation (globally) controllable?}

\subsection{Damping stabilization for  waves}
From now on we fix   $a(x)$ as  some function that is  non-negative, smooth, and supported in $\omega\subset \S$, which is assumed to have non-empty interior. We shall also assume that $a(x)$ is strictly positive in $\omega_0$, an interval inside $\omega$. 
Let us consider the damped wave maps equation, 
\begin{equation}\label{equation:dwm}
     \Box \phi= \left(|\phi_t|^2- |\phi_x|^2\right)\phi+ a(x) \phi_t, \; \phi[0]\in \mathbb{S}^k\times T\mathbb{S}^k.
\end{equation}
  The damped wave maps equation can be  regarded as a special case of the controlled wave maps equation. In particular it also remains $\mathbb{S}^k$-invariant whence $\phi_t\perp \phi$.  In fact, by labelling $y(t, x)$ the value of $|\phi(t, x)|^2$ we get 
\begin{gather*}
    \Box y= 2\left(|\phi_t|^2-|\phi_x|^2\right)(y- 1)+ a(x) y_t, \\
    y(0, x)=1 \textrm{ and } y_t(0, x)=0, \; \forall x\in \S,
\end{gather*}
hence $|\phi(t, x)|=1$.

It is well-known that  damping  dissipates waves' energy, we refer to  the (nonlinear) wave equations \cite{Anantharaman-Leautaud-2014, Bardos-Lebeau-Rauch, Dehman-Lebeau-Zuazua, Lagnese}, the defocusing Klein-Gordon equations \cite{Laurent-2011}, the (nonlinear) Schrödinger equations \cite{DGL-2006}, KdV \cite{zuazua02, rosier97}, among others. Indeed, this is also the case for wave maps:
\begin{align*}
    \frac{1}{2} \frac{d}{dt} E(t)
    = -\int_{\mathbb{S}^1}a(x)|\phi_t|^2(t, x) dx\leq 0,
\end{align*}
thus implying
\begin{equation*}
    E(T)- E(0)= -2\int_0^T\int_{\mathbb{S}^1}a(x)|\phi_t|^2(t, x)\, dx dt.
\end{equation*}
Based on this  energy dissipation observation, it is natural to ask about the stability of such damped systems: is the damped wave maps equation asymptotically  stable or even exponentially stable? In control theory this type of stabilization problem has been extensively studied, which turned out to be linked with \textit{unique continuation} problems and  \textit{observability inequality} type estimates. To be more precise, suppose that the following observability inequality holds,
\begin{equation}\label{eq:int:obin}
  E(0)\leq c \int_0^T\int_{\mathbb{S}^1}a(x)|\phi_t|^2(t, x) dxdt,
\end{equation}
then the (nonlinear) system is exponentially stable with some decay rate depending explicitly on $T$ and $c$:
\begin{equation*}
    E(t)\leq C e^{-\gamma t} E(0), \; \forall t\in (0, +\infty).
\end{equation*}

In the context of wave equations, damping terms lead to exponential stability provided GCC is satisfied, according to \cite{Bardos-Lebeau-Rauch}.
Otherwise weaker stability properties can be expected such as logarithmic stability, see for instance  \cite{Anantharaman-Leautaud-2014, Lebeau-Robbiano-stabilization} and the references therein on this subject.

By considering the nonlinear terms as perturbation, it is standard to derive from the stability of the damped wave equations that the ``linearized damped wave maps equation" is exponentially stable.  One may ask  the following question:

\noindent {\bf Question 2:}
    \textit{What is the global stability of the damped wave maps equation?}

\subsection{Main results and the strategy of proofs}

In this paper we answer \textit{Questions 1--2} concerning (global) controllability and stabilization of wave maps. More precisely, we  present the following theorem on the quantitative semi-global  controllability of wave maps, where by semi-global we mean that the energy is strictly below $2\pi$.    We will see that $2\pi$ is the energy threshold on  stabilization of the damped wave maps equation, moreover, it is also optimal for controllability if the wave maps' target is $\S$.  However, global controllability of wave maps is anticipated when the target is $\mathbb{S}^k$ with $k$ strictly greater than 1. To the best of our knowledge this is the first control result on a geometric wave equation.

\begin{theorem}[Semi-global exact controllability of   wave maps]\label{mainthm}
For any $\nu>0$ there exists some effectively computable $T\geq 2\pi$ and $C>0$ such that, for any pair of states $u[0]$ and $u[T]: \S\rightarrow \mathbb{S}^k\times T\mathbb{S}^k$ verifying 
\begin{equation*}
    \|u[0]\|_{\dot{H}^1_x\times L^2_x}, \|u[T]\|_{\dot{H}^1_x\times L^2_x}\leq 2\pi-\nu,
\end{equation*}
we are able to construct a $\mathbb{R}^{k+1}$-valued control $f(t, x)$ satisfying 
\begin{equation*}
     \|f\|_{L^{\infty}_t L^2_x([0, T]\times \S)}\leq C \left(\|u[0]\|_{\dot{H}^1_x\times L^2_x}+ \|u[T]\|_{\dot{H}^1_x\times L^2_x}\right) ,
\end{equation*}
such that the unique solution of the inhomogeneous wave maps equation
\begin{gather*}
     \Box \phi= \left(|\phi_{t}|^2- |\phi_{ x}|^2\right) \phi+ \mathbf{1}_{\omega} f^{\phi^{\perp}}, \; \phi[0]= u[0],
\end{gather*}
 verifies $(\phi, \phi_t)(t, x)\in \mathbb{S}^k\times T\mathbb{S}^k$ and
 \begin{gather*}
 \phi[T]= u[T].
 \end{gather*}
\end{theorem}

Here we adapt a two-step strategy for  semi-global controllability: \textit{global-stabilization/local-control}.
\begin{itemize}
    \item The first step is devoted to stabilize wave maps' energy using damping terms, as damping terms can be regarded as  special control terms. This process stops as soon as the energy becomes sufficiently small. See Theorem~\ref{thm:main} for details, whose proof is presented in Section~\ref{sec-stabilizationstep}.
    \item In the second step  we prove exact controllability of wave maps for states with small energy. Actually, this step is again composed of two stages: we first prove local exact controllability around every trivial equilibrium state, $i.e.$ $(p, 0)$ with $p\in \mathbb{S}^k$; then for any $p, q\in \mathbb{S}^k$ we move the state  from $(p, 0)$ to $(q, 0)$ slowly. Section~\ref{sec:controlstep} is devoted to this step leading to the proof of     Theorem~\ref{thm:smallcontrolwm}.
\end{itemize}
Such a \textit{global-stabilization/local-control} strategy has been widely adapted to global controllability problems.   It is natural to expect local controllability results in cases where linearized systems are controllable. However, when dealing with global problems, usually controllability around equilibria is not  enough to conclude global controllability properties, as nonlinear terms may play a significant role: that is the reason we first stabilize the system.     Damping provides a simple stabilization process for global stabilization of nonlinear systems, as we may observe from the inequality  \eqref{eq:int:obin} that exponential stability is hidden from observability inequalities despite nonlinear structure. Successful examples include but are not limited to Benjamin-Ono equations \cite{Laurent-Linares-Rosier-2015}, KdV equations \cite{Laurent-Rosier-Zhang-kdv}, NLW equations \cite{Dehman-Lebeau-Zuazua, Laurent-2011}, NLS equations \cite{DGL-2006}.\\  Concerning global stabilization there exist many  other  techniques in the literature, for instance, the authors in  \cite{Coron-trelat-2004}   have benefited from a  stabilization process to move  solutions around initial states to some point near final states; in \cite{Coron-Xiang-2018} based on the transport property of the viscous Burgers equations some feedback laws have been constructed to globally stabilize  states in short time.  We shall also mention that in some circumstances it is more efficient to  obtain global controllability without relying on stabilization. In particular, the return method introduced by Coron turned out to be a powerful philosophy  to get  small-time global controllability, as is the case for various models including Euler equations \cite{Coron-1996-Euler, Glass-2000}, Navier--Stokes equations \cite{1996-Coron-Fursikov-RJMP, Coron-Marbach-Sueur-Zhang}, KdV \cite{Chapouly-kdv},  the viscous Boussinesq system \cite{chavessilva2020smalltime} among others.  \\

\noindent {\bf Low-energy exact controllability of   wave maps}

Keeping in mind that wave equations are exactly controllable, 
 by ignoring the geometric constraints and regarding the controlled wave maps equation as a controlled semilinear wave system,  it sounds natural to expect local exact controllability around  trivial equilibrium points $(p, 0)$ with  $ p\in \mathbb{S}^k$.  The main difficulty comes from the geometric constraint  ensuring that  $u[t]$ remains on the target $ \mathbb{S}^k\times T\mathbb{S}^k$; this is handled by an iterative control construction.      

\begin{theorem}[Low-energy exact controllability of   wave maps]\label{thm:smallcontrolwm}
There exist effectively computable $T\geq 2\pi, e_s>0$ and $C_G>0$ such that, for any pair of states $u[0]$ and $u[T]: \S\rightarrow \mathbb{S}^k\times T\mathbb{S}^k$ verifying 
\begin{equation*}
    \|u[0]\|_{\dot{H}^1_x\times L^2_x}+ \|u[T]\|_{\dot{H}^1_x\times L^2_x}\leq e_s,
\end{equation*}
we are able to construct a $\mathbb{R}^{k+1}$-valued control function $f(t, x)$ satisfying 
\begin{equation*}
     \|f\|_{L^{\infty}_t L^2_x([0, T]\times \S)}\leq C_G \left(\|u[0]\|_{\dot{H}^1_x\times L^2_x}+ \|u[T]\|_{\dot{H}^1_x\times L^2_x}\right) ,
\end{equation*}
such that the unique solution of the inhomogeneous wave maps equation
\begin{gather*}
     \Box \phi= \left(|\phi_{t}|^2- |\phi_{ x}|^2\right) \phi+ \mathbf{1}_{\omega} f^{\phi^{\perp}}, \; \phi[0]= u[0],
\end{gather*}
 verifies $(\phi, \phi_t)(t, x)\in \mathbb{S}^k\times T\mathbb{S}^k$ and
 \begin{gather*}
 \phi[T]= u[T].
 \end{gather*}
\end{theorem}
\noindent Remark that \textit{low energy controllability} implies the typical \textit{local controllability}  but is different from  the latter.\\

\noindent {\bf Semi-global stabilization of the damped wave maps.}

Thanks to the preceding step, it suffices to prove semi-global approximate controllability, namely for any initial state we are able to steer it towards some state having sufficiently small energy. \\ Here, we focus on damping stabilization where damping terms can be regarded as a simple local feedback law,   despite various other powerful stabilization techniques in the literature. These methods usually provide more complicated nonlocal feedback laws, which include  but are not limited to, the backstepping method \cite{coron2021stabilization, coronluqi,Gagnon-Hayat-Xiang-Zhang}, the  basic Lyapunov approach for systems of conservation laws \cite{SaintVenantPI}, the frequency Lyapunov for finite time stabilization problems \cite{Xiang-heat-2020}, and \cite{krieger2020boundary} for the focusing NLKG.

Regarding  Theorem~\ref{mainthm} one may ask \textit{where the  $2\pi$-energy bound in Theorem~\ref{mainthm} comes from.}  In fact,
 harmonic maps appear as non-trivial stationary states for wave maps,  $i. e.$ functions $\varphi(x)$ satisfying 
\begin{equation*}
  \Delta \varphi=   \left(|\varphi_t|^2- |\varphi_x|^2\right)\varphi.
\end{equation*}
They are also stationary states for the damped wave maps equation:
 we are not able to expect any global ($i.e.$ large data) exponential stability of the damped wave maps equation. 
In our simplified context, namely maps from  $\S$ to $\mathbb{S}^k$, such a harmonic map is given by 
\begin{equation*}
    \mathcal{Q}(x)= (\cos{x}, \sin{x}, 0,..., 0)\in \mathbb{S}^{k}\subset \mathbb{R}^{k+1}, \; \forall x\in \mathbb{S}^1,
\end{equation*}
whose energy is
\begin{equation*}
    E= 2\pi.
\end{equation*}

Inspired by the preceding observations it is natural to ask {\it if  the damped wave maps equation is exponentially stable for initial states with energy uniformly below the threshold $2\pi$.}
The answer is positive:
\begin{theorem}[Semi-global stabilization of the damped wave maps]\label{thm:main}
For any $\nu>0$ there exist some effectively computable $C$ and $c$ such that for any initial state $\phi[0]: \S\rightarrow \mathbb{S}^k\times T\mathbb{S}^k$  satisfying 
\begin{equation*}
    E(0)\leq 2\pi- \nu
\end{equation*}
the unique solution of the damped wave maps equation
\begin{equation*}
      \Box \phi= \left(|\phi_t|^2- |\phi_x|^2\right)\phi+ a(x) \phi_t, 
\end{equation*}
decays exponentially:
\begin{equation*}
    E(t)\leq C e^{-c t} E(0), \; \forall t\in (0, +\infty).
\end{equation*}
\end{theorem}

The advantages of this stabilization result are twofold. \\
First, we would like to emphasize its quantitative feature. In fact the proof is based on a quantitative version of the observability inequality \eqref{eq:int:obin}. In the literature usually such observability inequalities are derived from   compactness/uniqueness arguments of contradiction type, and consequently,  in many circumstances we do not have enough knowledge on the exact value of those constants. This argument is standard  for stability problems in control theory including   KdV equations \cite{rosier97}, the semilinear wave equations \cite{Dehman-Lebeau-Zuazua, Laurent-2011}, the (nonlinear) Schrödinger equations \cite{DGL-2006} among others.  Is there a quantitative alternative to the  compactness/uniqueness argument? In  \cite{Krieger-Xiang-kdv} the authors have introduced a method for this purpose for KdV equations, where the proof is based on smoothing effects of semigroups and explicit flux estimates. Theorem~\ref{thm:main} provides another instance of bypassing compactness/uniqueness arguments.\\
Second, we shall mention that  this analysis framework is based on the nonlinear structure of wave maps. Indeed, as we will see in Section~\ref{sec:2:structure}, Proposition~\ref{prop:1w}, the  \textit{observation} becomes weaker  when the wave map's energy approaches the energy threshold $2\pi$ which perfectly fits the limitation on damping caused by harmonic maps. Such a phenomenon does not exist for the defocussing nonlinear wave equations.

\subsection{Some comments on global results}
\;

Let us briefly turn to the case when we replace the special target $\mathbb{S}^k$ by a general Riemannian target $\mathcal{M}$. By Nash's embedding theorem, we may assume that $\mathcal{M}$ is isomtrically embedded into some $\mathbb{R}^k$. 
Notice that solutions of the controlled wave maps equation belong to $C([0, T]; H^1_x \times L^2_x (\S))$ taking values in $\mathcal{M}$. Due to the Sobolev embedding $H^1(\S)\hookrightarrow C(\S)$, for any time $t\in [0, 1]$ the state $\phi(t\cdot T)\in H^1(\S)$ can be regarded as a curve (loop) on $\mathcal{M}$:
\begin{align*}
\gamma^t(x): \S&\rightarrow \mathcal{M}, \\
             x&\mapsto \gamma^t(x):= \phi(t\cdot T, x).
\end{align*}
This implies that $\gamma^0$ deforms to $\gamma^1$ continuously, namely $\gamma^t$ is a homotopy from $\gamma_0$ to $\gamma_1$. This gives a necessary condition on global exact controllability: \textit{the wave maps equation can only be globally exactly controllable if the fundamental group $\pi_1(\mathcal{M})$ is trivial.} As a direct consequence of this  condition we notice that  Theorem~\ref{mainthm} is optimal in the case with $k=1$; however, this is probably not the case for $k\geq 2$ despite the fact that harmonic maps appear as obstruction for global damping stabilization.

Moreover,  this homotopy information also motivates the  following conjecture:
\begin{conj}
Let $(\mathcal{M}, g)$ be a Riemannian manifold. The controlled wave maps equation $\phi: \mathbb{R}\times \S\rightarrow \mathcal{M}$ is globally exactly controllable in the homotopic sense. More precisely, for any pair of states $u[0]$ and $v[0]: \S\rightarrow \mathcal{M}\times T\mathcal{M}$ belonging to $H^1_x\times L^2(\S)$ such that $u(x)$ and $v(x)$ are homotopic, we are able to construct a $\mathbb{R}^k$-valued control $f\in L^2_{t, x}([0, T]\times \S)$ with some $T>0$ such that the unique solution of the controlled wave maps equation with the initial state $\phi[0]= u[0]$ and the control $f$ verifies $\phi[T]= v[0]$.
\end{conj}

Finally, it is important to comment that in this framework we choose to limit ourselves to large time controllability properties and leave the more refined \textit{small-time controllability} or \textit{optimal-time controllability} problems to later investigations, namely to control the states from one to another in some optimal time or any possible small time (when 0 is the optimal time). For example, the incompressible Euler equation is exact controllable in small time \cite{Coron-1996-Euler, Glass-2000}, while  the transport equation with boundary control is exact controllable if and only if the control time is larger than $L$,
$$y_t+ y_x=0, \; y(t, 0)= u(t), \; \forall x\in (0, L).$$  
Due to the finite speed of propagation it is not possible to get small-time controllability.  But maybe we can expect (global) controllability of wave maps in some optimal time period like  $2\pi$.

\section{Quantitative semi-global stabilization of the damped wave maps}\label{sec-stabilizationstep}
This section is devoted to the proof of semi-global stabilization of the damped wave maps equation. 
This  stability analysis  is based on the damping effect and the quantitative characterization of semi-global  observability inequalities \eqref{eq:int:obin}.   More precisely, in Section~\ref{sec:2:structure} we show that several  properties lead to the stabilization result, namely Propositions~\ref{prop:1w}--\ref{prop:3}. Then, after making some preliminary preparation in  Section~\ref{subsec:prel},  these propositions are successively proved in Sections~\ref{subsec:prop1}--\ref{subsec:prop3}.
\subsection{Strategy of the semi-global stabilization}\label{sec:2:structure}
The stabilization result, $i.e.$ Theorem~\ref{thm:main},  is a consequence the following three propositions.

\begin{prop}\label{prop:1w}
There exists some $c_b\in (0, 1)$ effectively computable such that, for any $\mu\in (0, \pi)$ and for any initial state $\varphi[0]: \S\rightarrow \mathbb{S}^k\times T\mathbb{S}^k$ satisfying 
\begin{gather}
    \mu\leq E(0)\leq 2\pi-\sqrt{\mu},
\end{gather}
the unique 
solution of the damped wave maps equation \eqref{equation:dwm},
\begin{equation*}
   \Box \phi= \left(|\phi_t|^2- |\phi_x|^2\right)\phi+ a(x) \phi_t, \;     \phi[0]= \varphi[0],
\end{equation*}
verifies
\begin{equation*}
   \|\phi_t\|_{L^{\infty}_x(\S; L^2_t(0, 3\pi))}^2> \delta E(0),
\end{equation*}
where $\delta= \delta(\mu)= c_b \mu>0$.
\end{prop}

\begin{prop}\label{prop:2}
There exist some $p>0$ and $C_p>0$ effectively computable such that, for any $\delta\in (0, 1)$ we have $\varepsilon_0= \varepsilon_0(\delta)= C_p \delta^p$ such that, if  some solution of the damped wave maps equation \eqref{equation:dwm} verifies
\begin{gather}
    E(0)\leq 2\pi,\\
    \int_{-16\pi}^{16\pi}\int_{\mathbb{S}^1}a(x)|\phi_t|^2(t, x) dxdt\leq \varepsilon_0  E(0), \label{as:eq:p2}
\end{gather}
then 
\begin{equation}
     \|\phi_t\|_{L^{\infty}_x(\S; L^2_t(0, 3\pi))}^2\leq \delta E(0).
\end{equation}
\end{prop}

\begin{prop}[Low-energy  exponential stability]\label{prop:3}
There exist some $\mu_0>0$ and $c_{\mu_0}$ effectively computable such that, if some  $\phi$, a solution of the damped wave maps equation \eqref{equation:dwm}, verifies
\begin{gather*}
     E(0)\leq \mu_0,
\end{gather*}
then 
\begin{equation*}
   c_{\mu_0} E(0)\leq  \int_{0}^{16\pi}\int_{\mathbb{S}^1}a(x)|\phi_t|^2(t, x) dxdt.
\end{equation*}
\end{prop}

 Let us quickly comment on how Propositions~\ref{prop:1w}--\ref{prop:3} lead to quantitative exponential decay of the damped wave maps equation, namely Theorem~\ref{thm:main}.  Let $Q>1$ be some given constant such that\footnote{It is easily seen that one may choose $Q$ only in dependence of $\|a\|_{L^\infty}$.} 
\begin{equation}
   Q^{-1} E(-16\pi)\leq E(0)\leq  E(-16\pi).
\end{equation}
 
\begin{proof}[Proof of Theorem~\ref{thm:main}]
We first show that  Proposition~\ref{prop:1w} together with Proposition~\ref{prop:2} implies the asymptotic stability.  Indeed, 
if $E(-16\pi)\in [\mu, 2\pi-\sqrt{\mu}]$, then $E(0)\in [Q^{-1}\mu, 2\pi-\sqrt{\mu}]$. Thanks to Proposition~\ref{prop:1w} we know that 
\begin{equation*}
   \|\phi_t\|_{L^{\infty}_x(-\pi, \pi; L^2_t(0, 3\pi))}^2> \delta E(0),
\end{equation*}
with $\delta=\delta(Q^{-1}\mu)= c_b Q^{-1} \mu$. Next, according to Proposition~\ref{prop:2}, we must have that 
\begin{equation*}
     \int_{-16\pi}^{16\pi}\int_{\mathbb{S}^1}a(x)|\phi_t|^2(t, x) dxdt> \varepsilon_0  E(0),
\end{equation*}
with $\varepsilon_0= \varepsilon_0(\delta)= C_p (c_b Q^{-1}\mu)^p$. Thus, for any $E(-16\pi)\in [\mu, 2\pi-\sqrt{\mu}]$ we have 
\begin{equation*}
     \int_{-16\pi}^{16\pi}\int_{\mathbb{S}^1}a(x)|\phi_t|^2(t, x) dxdt>   C \mu^p E(-16\pi),
\end{equation*}
which implies that, for any $E(-16\pi)\in (0, \pi)$ the following holds:
\begin{equation*}
    E(- 16\pi)- E(16 \pi)= 2\int_{-16\pi}^{16\pi}\int_{\mathbb{S}^1}a(x)|\phi_t|^2(t, x) dxdt>   C E(-16\pi)^{p+1}.
\end{equation*}
This is easily seen to imply asymptotic stability. 
\\

Finally we invoke Proposition~\ref{prop:3} to conclude the desired semi-global exponential stability. Take $\mu$ as $\mu_0$ in the preceding step. 
On the one hand, if $E(-16\pi)\in [\mu_0, 2\pi-\sqrt{\mu_0}]$, then 
\begin{equation*}
     \int_{-16\pi}^{16\pi}\int_{\mathbb{S}^1}a(x)|\phi_t|^2(t, x) dxdt>   C \mu_0^p E(-16\pi).
\end{equation*}
On the other hand, if $E(-16\pi)\in [0, \mu_0]$, then according to Proposition~\ref{prop:3} 
\begin{equation*}
   \int_{-16\pi}^{16\pi}\int_{\mathbb{S}^1}a(x)|\phi_t|^2(t, x) dxdt\geq  c_{\mu_0} E(-16\pi).
\end{equation*}
Hence 
\begin{equation*}
    E(16\pi)\leq \big(1- 2\min\{c_{\mu_0},  C \mu_0^p\}\big) E(-16\pi).
\end{equation*}
\end{proof}

\subsection{Some preliminaries}\label{subsec:prel}
 Without loss of generality we may assume that $a(x)$ is greater than $1$  inside a small interval $\omega_0$, and we may also assume that $\omega_0= (-l_0/2, l_0/2)$.
As usual we interchange the roles of $t$ and $x$, and focus on the diamond shaped region thanks to the finite speed of propagation, 
\begin{equation*}
    D:= \{(t, x): |t+x|\leq 18\pi, |t-x|\leq 18\pi\}
\end{equation*}
which covers the basic domain of interest 
\begin{equation*}
  D_0:= \{(t, x): x\in [-2\pi, 2\pi], t\in [-16\pi, 16\pi]\}.
\end{equation*}
For  ease of presentation we also define 
\begin{equation*}
  D_1:= \{(t, x): x\in [-18\pi, 18\pi], t\in [-18\pi, 18\pi]\},
\end{equation*}
such that $D_0\subset D\subset D_1$.  
Throughout this paper we  focus on (as we may) solutions at energy level regularities, but only work with   $C^{\infty}(D)$  solutions to make sure all operations in the sequel are justified. In other words, for any given initial state
$$(\phi(0, x), \phi_t(0, x))\in H^1_x(\mathbb{S}^1)\times L^2_x(\mathbb{S}^1),$$ 
we study
\begin{gather*}
    \phi(t, x)\in H^1(D),\\
    (\phi_t(t, x), \phi_{t x}(t, x))\in L^2_t(0, T)\times H^{-1}_t(0, T),  \, \forall x\in [-\pi, \pi].
\end{gather*}

\subsubsection{Null coordinates}
In terms of the standard null coordinates
\begin{equation}\label{eq:nullcoor}
    u= x+t,\; v= t-x,
\end{equation}
 the diamond shaped region becomes 
\begin{equation}
    D= \{(u, v): |u|\leq 18\pi, |v|\leq 18\pi\}. \notag
\end{equation}
In the following  we  keep  the notation $\phi$ for functions in null coordinates, $i.e.$ $\phi(t, x)= \phi(u, v)$ provided the relation \eqref{eq:nullcoor}. Thus 
\begin{gather*}
    \phi_u(u, v)= \phi_u(t, x)= \frac{1}{2}\left(\phi_t+ \phi_x\right)(t, x), \; \; \phi_v(t, x)= \frac{1}{2}\left(-\phi_x+ \phi_t\right)(t, x),  \\
    \phi_u\cdot \phi_v= \frac{1}{4} (-|\phi_x|^2+ |\phi_t|^2), \; \;
    \phi_{uv}= \frac{1}{4}(\phi_{tt}- \phi_{xx}).
\end{gather*}
Notice that the equation for wave maps is reduced to
\begin{equation}\label{eq:wmuv}
     -\phi_{uv}= (\phi_u\cdot \phi_v) \phi+ \frac{1}{4} a \phi_t=: F(t, x)= F(u, v).
\end{equation}
As an immediate consequence of this change of variables one can express $\phi$ as follows. 
 By denoting $u_0= y+s, v_0= s-y$, so that $(s, y)$ corresponds to the original $(t,  x)$ coordinates while  $(u_0, v_0)$ corresponds to the $(u, v)$ coordinates,
 we integrate on the triangular region to benefit from the information on the line $t=0$,
\begin{align*}
     \phi_u(u, v)= -\int_{-u}^v F(u, v_0) d v_0+ \phi_u(u, -u)
     = -2\int_{0}^{t} F(s, x+t-s) ds+ \phi_u(u, -u),
 \end{align*}
 and $\phi$ is further given by
 \begin{align*}
    \phi(u, v)&= \int_{-v}^u \int_{-u_0}^v -F(u_0, v_0) dv_0 du_0+ \int_{-v}^u \phi_u(u_0, -u_0) du_0+ \phi(-v, v),\\
    &= 2\int_{-v}^u \int_{-u_0}^v F(s, y) dy ds+ \int_{-v}^u \phi_u(u_0, -u_0) du_0+ \phi(-v, v),\\
    &= 2\int_0^t \int_{x-t+s}^{x+t-s} F(s, y) dy ds+ \int_{-v}^u \phi_u(u_0, -u_0) du_0+ \phi(-v, v).
\end{align*}

For the purpose of controlling the $H^{-1}$-norm of $\phi_{tx}$, we are also interested in the equation satisfied by $\phi_t(t,x)$:
\begin{equation}
    -\phi_{t, tt}+ \phi_{t, xx}= (|\phi_t|^2-|\phi_x|^2)\phi_t+ 2(\phi_t\cdot \phi_{tt}- \phi_x\cdot \phi_{tx})\phi+ a \phi_{tt}. \notag
\end{equation}
Since 
\begin{equation*}
    \phi_u \cdot \psi_v+ \phi_v\cdot \psi_u= \frac{1}{2}(\phi_t\cdot \psi_t-\phi_x\cdot \psi_x),
\end{equation*}
we get 
\begin{equation*}
    -\phi_{t, uv}= (\phi_u\cdot \phi_v)\phi_t+ (\phi_u\cdot\phi_{t, v}+ \phi_v\cdot \phi_{t, u})\phi+ \frac{1}{4}a \phi_{tt}=: G(t, x).
\end{equation*}
Similarly, $\phi_t$ is expressed by 
\begin{equation*}
    \phi_t(t, x)= 2\int_0^t \int_{x-t+s}^{x+t-s} G(s, y) dy ds+ \int_{-v}^u \phi_{t, u}(u_0, -u_0) du_0+ \phi_t(-v, v).
\end{equation*}

\subsubsection{Basic estimates}\label{sec-basices}
Here  we present several straightforward and useful quantitative energy estimates on the damped wave maps equation \eqref{equation:dwm}. Furthermore, by adapting the same proof,  similar estimates can be obtained to the inhomogeneous wave maps equation  \eqref{eq:inhomowavemap} leading to Lemma~\ref{lem:inhwm} (see Section~\ref{sec:strlec} for details).
\begin{gather}
  E(0)\lesssim E(t)\lesssim E(0), \; \forall t\in (-36\pi, 36\pi), \label{es:e1}\\
    \|(\phi_t, \phi_x, \phi_u, \phi_v)\|_{L^2_{t, x}(D)}^2\lesssim E(0), \label{es:e2}\\
    \|(\phi_t, \phi_x)\|_{L^{\infty}_t L^2_x(D_1)}^2\lesssim E(0), \label{es:e32} \\
    \|(\phi_t, \phi_x)\|_{L^{\infty}_x L^2_t(D_1)}^2\lesssim E(0), \label{es:e3}  \\
    \|\phi_u\|_{L^2_u L^{\infty}_v(D)}^2+  \|\phi_v\|_{L^2_v L^{\infty}_u(D)}^2+  \|\phi_u\cdot \phi_v\|_{L^2_{u, v}(D)}\lesssim E(0).  \label{es:e4}
\end{gather}

\begin{proof}
For any given $t\in \R$, the estimates \eqref{es:e1} and \eqref{es:e32} are  direct consequences of the time variation of the energy. This further yields the $L^2_{t, x}$ bound in \eqref{es:e2}, while the $L^2_{u, v}$ bound follows directly.
Now we turn to the proof of the estimate \eqref{es:e3}. Thanks to \eqref{es:e2} we can pick some $\bar x\in (-\pi, \pi)$ such that 
\begin{equation*}
    \int_{-36\pi}^{36\pi} \left(|\phi_t|+ |\phi_x|\right)^2(t, \bar x)dt\lesssim E(0).
\end{equation*}
Consider the propagation of the ``vertical energy" with respect to $x$: by defining 
\begin{equation*}
    \tilde{E}(y):= \int_{-36\pi+|y|}^{36\pi-|y|}  \left(|\phi_t|+ |\phi_x|\right)^2(t, \bar x+ y)dt
\end{equation*}
thanks to integration by parts, we have 
\begin{align*}
 \frac{d}{dy}\tilde{E}(y)&= 2\int_{-36\pi+|y|}^{36\pi-|y|}  \left(\phi_t\cdot \phi_{tx}+ \phi_x\cdot \phi_{xx}\right)(t, \bar x+ y)dt \\
 &\;\;\;\;- \sign(y) \left(\left(|\phi_t|+ |\phi_x|\right)^2(36\pi-|y|, \bar x+ y)+ \left(|\phi_t|+ |\phi_x|\right)^2(-36\pi+|y|, \bar x+ y)\right) \\
 &= 2\int_{-36\pi+|y|}^{36\pi-|y|}  \left( \phi_x\cdot \Box \phi\right)(t, \bar x+ y)dt+ 2 (\phi_t\cdot \phi_x)\Big{|}_{-36\pi+|y|}^{36\pi-|y|} (t, \bar x+ y) \\
 &\;\;\;\;- \sign(y) \left(\left(|\phi_t|+ |\phi_x|\right)^2(36\pi-|y|, \bar x+ y)+ \left(|\phi_t|+ |\phi_x|\right)^2(-36\pi+|y|, \bar x+ y)\right).
\end{align*}
By plugging the wave maps equation into the preceding formula we obtain 
\begin{equation*}
   \sign(y) \frac{d}{dy}\tilde{E}(y)\lesssim \tilde{E}(y),
\end{equation*}
which implies the inequality \eqref{es:e3}.\\

Finally, concerning  \eqref{es:e4}  it suffices to prove the first two inequalities as the last estimate follows as a direct consequence. Thanks to \eqref{es:e2} there exists some $\bar v\in (-18\pi,  18\pi)$ such that 
\begin{equation*}
    \int_{-18\pi}^{18\pi} |\phi_u|^2(u, \bar v) du\lesssim E(0).
\end{equation*}
For  ease of notations we may assume that $\bar v=-18\pi$.  Recalling the wave maps equation under $(u, v)$ coordinate. we get 
\begin{equation*}
    \frac{d}{dv} |\phi_u|^2= 2 \phi_u\cdot \phi_{uv}= -\frac{a}{2} \phi_u\cdot \phi_t.
\end{equation*}
Thus, for any  $u, v\in (-18\pi, 18\pi)$ there is 
\begin{align*}
    |\phi_u(u, v)|^2&\lesssim \left(|\phi_u|(u, -18\pi)+ \int_{-18\pi}^{v}|\phi_t(u, v_0)|d v_0 \right)^2 \\
    &\lesssim |\phi_u|^2(u, -18\pi)+ \int_{-18\pi}^{v}|\phi_t(u, v_0)|^2d v_0. 
\end{align*}
Hence, for  any $v\in  (-18\pi, 18\pi)$
\begin{align*}
     \int_{-18\pi}^{18\pi} |\phi_u|^2(u,  v) du &\lesssim \int_{-18\pi}^{18\pi} \left( |\phi_u|^2(u, -18\pi)+ \int_{-18\pi}^{v}|\phi_t(u, v_0)|^2d v_0 \right)du \\
     &\lesssim E(0).
\end{align*}
Similarly one concludes  estimates on $\|\phi_v\|_{L^2_v L^{\infty}_u}$. 
\end{proof}

\subsection{Proof of Proposition~\ref{prop:1w}}\label{subsec:prop1}
This section is devoted to the proof of Proposition~\ref{prop:1w}, which is divided into four steps.  Let us construct a non-negative  smooth cutoff function $\psi(t)$  such that 
\begin{equation}
    \int_{\R} \psi(t)=1\;  \textrm{ and } \; \psi\;  \supp\; (0, 3\pi).
\end{equation}

\noindent {\it Step 1:} 
 Assume that for some $\delta\in (0, 1)$ we have 
\begin{gather}
    E(0)\leq 2\pi, \notag \\
    \|\phi_t\|_{L^{\infty}_x(-\pi, \pi; L^2_t(0, 3\pi))}^2\leq \delta E(0). \label{eq:ass:delta}
\end{gather}
This assumption implies that 
\begin{gather*}
 \int_{\R}\int_{\S} |\phi_t|^2(t, x) \psi(t)dx dt\lesssim   \delta E(0), \\
    \int_0^{3\pi}\int_{\mathbb{S}^1}a(x)|\phi_t|^2(t, x) dxdt\lesssim \int_0^{3\pi}\int_{\mathbb{S}^1}|\phi_t|^2(t, x) dxdt\lesssim \delta E(0),
\end{gather*}
thus
\begin{equation}
    E(0)- E(t)\lesssim \delta E(0), 
    \;  \forall t\in [0, 3\pi]. \notag
\end{equation}

By  definition of the energy
\begin{equation*}
    E(t)= \int_{\S} (|\phi_t|^2+ |\phi_x|^2)(t, x) \,dx,
\end{equation*}
we know that 
\begin{align*}
     \left|\int_{\R}\int_{\S} (|\phi_t|^2+ |\phi_x|^2)(t, x) \psi(t)dx dt- E(0)\right|= \left|\int_{\R} \psi(t)\big(E(t)- E(0)\big) dt\right|\lesssim   \delta E(0).
\end{align*}
Hence
\begin{align*}
     \left|\int_{\R}\int_{\S}  |\phi_x|^2(t, x) \psi(t)dx dt- E(0)\right|\lesssim   \delta E(0),
\end{align*}
which further implies the existence of some $\bar x\in \S$ such that 
\begin{align}
     \left|\int_{\R}  |\phi_x|^2(t, \bar x) \psi(t) dt- \frac{E(0)}{2\pi}\right|\lesssim   \delta E(0).
\end{align}

\noindent {\it Step 2:} In this step we show that the preceding inequality  holds for every $x\in \S$.  Write the wave maps equation as follows
\begin{equation*}
 \phi_{xx} + |\phi_x|^2\phi = \phi_{tt} + |\phi_t|^2\phi + a(x)\phi_t.
\end{equation*}
 Multiply the preceding wave maps equation  by $\phi_x$ and integrate against $\psi(t)$ in $(t, x)\in \R\times (\bar x, x_1)$. It follows that
\begin{align*}
 \int_{\bar x}^{x_1}\int_{\R}\phi_{xx}\cdot\phi_x\psi(t)\,dt dx
&= \int_{\bar x}^{x_1}\int_{\R}(\phi_{tt}\cdot\phi_x\psi(t) + a(x)\phi_t\cdot\phi_x\psi(t))\,dt dx\\
& =  \int_{\bar x}^{x_1}\int_{\R}(-\phi_{t}\cdot\phi_x\psi'(t) - \phi_t\cdot\phi_{xt}\psi(t) + a(x)\phi_t\cdot\phi_x\psi(t))\,dt dx.
\end{align*}
Because 
\begin{gather*}
    2\int_{\bar x}^{x_1} \phi_{xx}\cdot\phi_x(t, x) \, dx= |\phi_x|^2(t, x_1)- |\phi_x|^2(t, \bar x), \\
     2\int_{\bar x}^{x_1} \phi_{xt}\cdot\phi_t(t, x) \, dx= |\phi_t|^2(t, x_1)- |\phi_t|^2(t, \bar x),
\end{gather*}
we have 
\begin{align*}
\int_{\R} |\phi_x|^2(t,x_1)\psi(t)\, dt- \int_{\R} |\phi_x|^2(t,\bar x)\psi(t)\, dt
& =  2\int_{\bar x}^{x_1}\int_{\R}(-\phi_{t}\cdot\phi_x\psi'(t)  + a(x)\phi_t\cdot\phi_x\psi(t))\,dt dx \\
& \;\;\;\;\;  + \int_{\R} |\phi_t|^2(t, \bar x)\psi(t)\,dt - \int_{\R} |\phi_t|^2(t,x_1)\psi(t)\,dt.
\end{align*}
Keeping in mind that 
\begin{equation*}
    \int_{\S}\int_{0}^{3\pi}|\phi_{t}\cdot\phi_x|(t, x)\,dt dx\lesssim \sqrt{\delta} E(0),
\end{equation*}
we get
\begin{align}
     \left|\int_{\R}  |\phi_x|^2(t, x_1) \psi(t) dt- \frac{E(0)}{2\pi}\right|\lesssim   \sqrt{\delta} E(0), \; \forall x_1\in \S.
\end{align}

\noindent {\it Step 3:} Now we consider the following function $\tilde{\phi}(x)$ as well as the equation satisfied by it:
\begin{align*}
\tilde{\phi}(x): = \int_{\R} \phi(t, x)\psi(t)\,dt. 
\end{align*}
Note that $|\phi(t,x)| = 1$ and $\|\phi_t\|_{L_x^\infty L_t^2} = O(\sqrt{\delta E(0)})$, we have $|\tilde{\phi}| = 1+O(\sqrt{\delta E(0)})$. Indeed, for any $t, \bar t\in [0, 3\pi]$ and any $x\in \S$,
\begin{equation*}
    |\phi(t, x)- \phi(\bar t, x)|= \left|\int_{\bar t}^{t}\phi_t(s, x)\, ds\right|\lesssim \sqrt{\delta E(0)}.
\end{equation*}
Thus
\begin{equation*}
    |\tilde{\phi}(x)- \phi(\bar t, x)|\lesssim \sqrt{\delta E(0)}, \; \forall \bar t\in [0, 3\pi].
\end{equation*}
In particular, there exists  some $\bar \delta\in (0, 1)$ such that for any $\delta\in (0, \bar \delta]$ and for any $E(0)\in (0, 2\pi]$ satisfying the assumption \eqref{eq:ass:delta}, there is 
\begin{equation}\label{eq:tilphirange}
    |\tilde \phi(0)|\in \left(\frac{1}{2}, \frac{3}{2}\right).
\end{equation}

By the definition of $\tilde{\phi}$ and the wave maps equation,
\begin{align*}
    \tilde{\phi}_{xx}(x)&=  \int_{\R} \phi_{xx}(t, x)\psi(t)\,dt \\
    &= \int_{\R} \left(-|\phi_x|^2\phi+ \phi_{tt} + |\phi_t|^2\phi + a(x)\phi_t\right)(t, x) \psi(t)\,dt \\
    &= \int_{\R} -|\phi_x|^2\phi\psi(t)- \phi_{t}\psi_t(t) + (|\phi_t|^2\phi + a(x)\phi_t) \psi(t)\,dt.
\end{align*}
Successively, we are able to derive that for  $\forall x\in \S$,
\begin{gather*}
    \left|\int_{\R} - \phi_{t}\psi_t(t) + (|\phi_t|^2\phi + a(x)\phi_t) \psi(t)\,dt\right|\lesssim \sqrt{\delta E(0)}, \\
    \left|\int_{\R} (|\phi_x|^2\phi)(t, x)\psi(t)\,dt- \tilde{\phi}(x)\int_{\R}|\phi_x|^2(t, x)\psi(t)\,dt\right|\lesssim  \sqrt{\delta E(0)} E(0),
\end{gather*}
and
\begin{align*}
 \left|\tilde{\phi}(x)\int_{\R}|\phi_x|^2(t, x)\psi(t)\,dt- \frac{E(0)}{2\pi} \tilde{\phi}(x)\right| \lesssim  \sqrt{\delta E(0)}. 
\end{align*}
Consequently
\begin{equation}
    \left|\tilde{\phi}_{xx}(x)+ \frac{E(0)}{2\pi} \tilde{\phi}(x)\right|\lesssim  \sqrt{\delta E(0)}, \; \forall x\in \S.
\end{equation}

\noindent {\it Step 4:}  Finally, in this step we show that for $\delta$ sufficiently small we must have  $E(0)=0$ which leads to a contradiction.
Let us denote by $c_E$ the value of $E(0)/ 2\pi$. Then
\begin{equation*}
    |f(x)|:= \left|\tilde{\phi}_{xx}(x)+ c_E \tilde{\phi}(x)\right|\lesssim  \sqrt{\delta E(0)}, \; \forall x\in \S.
\end{equation*}
By decomposing 
\begin{gather*}
    f(x)= \sum_{n\in \mathbb{Z}} f_n e^{i nx}  \; \textrm{ and } \; 
     \tilde{\phi}(x)= \sum_{n\in \mathbb{Z}} a_n e^{i nx},
\end{gather*}
we get 
\begin{equation*}
    a_n (c_E- n^2)= f_n, \; \forall n\in \mathbb{Z}.
\end{equation*}
 Remark that by working with wave maps   the notations $a_n$ and $f_n$ are referring to $k+1$-dimensional vectors.
We immediately notice from the definition of $f_n$ that
\begin{equation*}
    |f_n|\lesssim  \sqrt{\delta E(0)}, \; \forall n\in \mathbb{Z},
\end{equation*}
which further implies
\begin{gather*}
    |a_n|= \frac{f_n}{c_E- n^2}\lesssim  \frac{1}{n^2} \sqrt{\delta E(0)}, \; \forall n\in \mathbb{Z}\setminus\{0, \pm 1\}, \\
    |a_0|= \left|\frac{f_0}{c_E}\right|\lesssim \sqrt{\frac{\delta}{E(0)}} \; \textrm{ and } \;
    |a_{\pm 1}|= \left|\frac{2\pi f_{\pm 1}}{E(0)- 2\pi}\right| \lesssim \frac{\sqrt{\delta E(0)}}{2\pi- E(0)}.
\end{gather*}

Let us assume that for some $\mu\in (0, \pi)$ the initial energy further verifies
\begin{equation}
    E(0)\in [\mu, 2\pi- \sqrt{\mu}].
\end{equation}
Under this assumption there is
\begin{equation*}
    |\tilde \phi(0)|\leq \sum_{n\in \mathbb{Z}}|a_n|\leq C \left(\sqrt{\delta E(0)}+ \sqrt{\frac{\delta}{E(0)}}+ \sqrt{\frac{\delta}{\mu}}\right).
\end{equation*}
Define $c_b$ in such fashion that $2\pi c_b<\bar \delta$ and that
\begin{equation*}
    C \left(\sqrt{2c_b \pi^2}+ 2\sqrt{c_b}\right)\leq \frac{1}{2}.
\end{equation*}
Then, by choosing $\delta= c_b \mu\in (0, \bar \delta)$ we know that 
\begin{equation*}
    |\tilde \phi(0)|\leq \frac{1}{2},
\end{equation*}
which is in contradiction with \eqref{eq:tilphirange}. Thus we conclude the proof of Proposition~\ref{prop:1w}.

\subsection{Proof of Proposition~\ref{prop:2}}
This section is devoted to the proof of Proposition~\ref{prop:2}. The strategy  is to  first find some point $x_0$ such that $\|\phi_t(t, x_0)\|_{L^2_t(-T, T)}$ is sufficiently small, and then to propagate such smallness to $x\in (-\pi, \pi)$ (or $x\in (0, 2\pi)$). More precisely, in Section~\ref{subsec:choicex0} we prove Lemma~\ref{lemma:choose} concerning the choice of $x_0$;  Section~\ref{subsec:propsmall} is devoted to the proof of Lemma~\ref{lem:keypropa} on  quantitative characterization of the propagation of smallness; finally, thanks to these two lemmas, in Section~\ref{subsec:propro2} we conclude the proof of  Proposition~\ref{prop:2}.\\

\noindent {\bf Cutoff functions}

At first  we introduce a series of cutoff functions which will be used later on especially when dealing with $H^{-1}_t$-norms. Select a even, smooth, non-negative,  truncated function $\mu(t)$ such that 
\begin{equation*}
    \mu(t)=1 \textrm{ on } [0, 1/2], \mu(t)=0 \textrm{ on } [1, +\infty].
\end{equation*}
For any $\alpha<\beta$ and for any $\tau\in (0, 1)$ we define the following cutoff function $\eta_{\alpha}^{\beta}[\tau]$ as
\begin{equation}
    \eta_{\alpha}^{\beta}[\tau](t):= 
    \begin{cases}
    \mu(\frac{t-\beta}{\tau}), \;\; \forall t\in (\beta, +\infty), \\
    1, \;\; \forall t\in [\alpha, \beta], \\
    \mu(\frac{t-\alpha}{\tau}), \;\;  \forall t\in (-\infty, \alpha), 
    \end{cases}
\end{equation}
that is supported in $(\alpha- \tau, \beta +  \tau)$.

\subsubsection{\bf On the choice of $x_0$.}\label{subsec:choicex0}

We are interested in both $\|\phi_t(\cdot, x)\|_{L^2_t}$ and $\|\phi_{tx}(\cdot, x)\|_{H^{-1}_t}$. To avoid possible problems on the definition of $H^{-1}_t$ norms, we work on the truncation functions, $\eta_{\alpha}^{\beta}[\tau] \phi_{tx}$, 
and  estimate 
\begin{equation*}
     \|\langle \partial_t\rangle^{-1} \big(\eta_{\alpha}^{\beta}[\tau](t) \phi_{tx}(t, x)\big)\|_{L^2_t(\mathbb{R})}.
\end{equation*}

\begin{lemma}\label{lemma:choose}
There exists some effectively computable $C_0>0$ such that, for any  $\varepsilon\in (0, 1)$, for any $\tau\in (0, 1)$,
if there is some solution  $\phi$  of the damped wave maps equation satisfying 
\begin{gather}
    E(0)\leq 2\pi, \notag \\
    \int_{-16\pi}^{16\pi}\int_{\mathbb{S}^1}a(x)|\phi_t|^2(t, x) dxdt\leq \varepsilon  E(0), \label{es:mainasum} 
\end{gather}
then,  there exists some $x_0\in [0, 2\pi)$ such that
\begin{equation}\label{es:lem_31}
     \|\phi_t(t, x_0)\|_{L^2_t(-16\pi, 16\pi)}^2+   \|\langle\partial_t\rangle^{-1}\left(\eta_{-15\pi}^{15\pi}[\tau] \phi_{tx}\right)(t, x_0)\|_{L^2_t(\R)}^2 \leq C_0 \frac{\sqrt{\varepsilon}}{\tau^2} E(0).
\end{equation}
\end{lemma}

\begin{proof}[Proof of Lemma~\ref{lemma:choose}]

For ease of notation, in the remainder of this proof, we simply denote the function  $\eta_{-15\pi}^{15\pi}[\tau](t, x)$ by $\eta(t, x)$ or simply $\eta(t)$.
The assumption  \eqref{es:mainasum} is equivalent to
\begin{equation}\label{es:sm1}  
\int_{\mathbb{S}^1} a(x) \int_{-16\pi}^{16\pi} |\phi_t|^2(t, x) dt dx\leq \varepsilon E(0).
\end{equation}
Thus it suffices to conclude the smallness of 
\begin{equation*}  
\int_{\mathbb{S}^1} a(x) \int_{\mathbb{R}} \left(\langle \partial_t\rangle^{-1} \left(\eta(t) \phi_{tx}(t, x)\right)\right)^2 dt dx.
\end{equation*}
A priori we do not have smallness of the preceding candidate: therefore, more careful estimates are required to achieve it. \\

Observe that
\begin{equation*}
    \langle \partial_t\rangle^{-1} \left(\eta(t) \psi_{t}(t)\right)=  \langle \partial_t\rangle^{-1} \left(\partial_t(\eta \psi)- \eta_t \psi \right),
\end{equation*}
which, combined with the fact that $\eta(t)= \eta_{-15\pi}^{15\pi}[\tau](t)$, implies   
\begin{equation*}
   \|\langle \partial_t\rangle^{-1} \left(\eta(t) \psi_{t}(t)\right)\|_{L^2_t(\R)} \lesssim  \frac{1}{\tau}\|\psi\|_{L^2_t(-16\pi, 16\pi)}. 
\end{equation*}
Hence,
\begin{gather*}
   \|\langle \partial_t\rangle^{-1} \left(\eta(t) \phi_{tx}(t, x)\right)\|_{L^2_{t, x}(\R\times \S)}^2 \lesssim \frac{1}{\tau^2} \|\phi_x\|_{L^2_{t, x}(D_0)}^2\lesssim \frac{1}{\tau^2}E(0), \\
    \|\langle \partial_t\rangle^{-2} \left(\eta(t) \phi_{tt}(t, x)\right)\|_{L^2_t(\R)}\leq \|\langle \partial_t\rangle^{-1} \left(\eta(t) \phi_{tt}(t, x)\right)\|_{L^2_t(\R)} \lesssim \frac{1}{\tau} \|\phi_t\|_{L^2_t(-16\pi, 16\pi)}. 
\end{gather*}

In order to benefit from the smallness of \eqref{es:mainasum} we adapt integration by parts to consider 
\begin{align*}
&\;\;\;\; \int_{\mathbb{S}^1}  \int_{\mathbb{R}} a(x) \left(\langle \partial_t\rangle^{-1} \left(\eta(t) \phi_{tx}(t, x)\right)\right)\cdot  \left(\langle \partial_t\rangle^{-1} \left(\eta(t) \phi_{tx}(t, x)\right)\right) dt dx \\
&=  \int_{\mathbb{S}^1}  \int_{\mathbb{R}} a(x) \left(\langle \partial_t\rangle^{-1} \left(\eta(t) \phi_{tt}\right)\right)\cdot  \left(\langle \partial_t\rangle^{-1} \left(\eta(t) \phi_{xx}\right)\right) dt dx \\
&\;\;\;+ \int_{\mathbb{S}^1}  \int_{\mathbb{R}} a(x) \left(\langle \partial_t\rangle^{-1} \left(\eta_t(t) \phi_{t}\right)\right)\cdot  \left(\langle \partial_t\rangle^{-1} \left(\eta(t) \phi_{xx}\right)\right) dt dx \\
&\;\;\;+ \int_{\mathbb{S}^1}  \int_{\mathbb{R}} a(x) \left(\langle \partial_t\rangle^{-1} \left(\eta(t) \phi_{t}\right)\right)\cdot  \left(\langle \partial_t\rangle^{-1} \left(\eta_t(t) \phi_{xx}\right)\right) dt dx \\
&\;\;\;- \int_{\mathbb{S}^1}  \int_{\mathbb{R}} a_x(x) \left(\langle \partial_t\rangle^{-1} \left(\eta_t(t) \phi_{t}\right)\right)\cdot  \left(\langle \partial_t\rangle^{-1} \left(\eta(t) \phi_{tx}\right)\right) dt dx \\
&=: I+ II+ III+ IV.
\end{align*}
In the following we shall treat the preceding terms one by one. Firstly, noticing the formula
\begin{equation*}
    \frac{a_x(x)}{\sqrt{a(x)}}= 2\partial_x \sqrt{a(x)}\lesssim 1,
\end{equation*}
we immediately get the smallness of $IV$:
\begin{align*}
    IV&\leq \left\|\sqrt{a(x)} \langle \partial_t\rangle^{-1} \left(\eta_t(t) \phi_{t}\right)\right\|_{L^2_x(\S; L^2_t(\R))} \left\|\frac{a_x(x)}{\sqrt{a(x)}} \langle \partial_t\rangle^{-1} \left(\eta(t) \phi_{tx}\right)\right\|_{L^2_x(\S; L^2_t(\R))}\\
    &\lesssim \frac{1}{\tau} \left\|\sqrt{a(x)}  \left(\eta_t(t) \phi_{t}\right)\right\|_{L^2_x(\S; L^2_t(\R))}  \left\| \phi_{x}\right\|_{L^2_x(\S; L^2_t(-16\pi, 16\pi))}\\
    &\lesssim \frac{\sqrt{\varepsilon}}{\tau^2} E(0).
\end{align*}

As for the other components, $I-III$, the difficult part is on the estimates of  right hand side items:  $\langle \partial_t\rangle^{-1} \left(\eta(t) \phi_{xx}\right)$ and $\langle \partial_t\rangle^{-1} \left(\eta_t(t) \phi_{xx}\right)$. Keeping in mind that 
\begin{align*}
    \left\|\sqrt{a} \langle \partial_t\rangle^{-1} \left(\eta \phi_{tt}\right)\right\|_{L^2_x(\S; L^2_t(\R))}&\lesssim \frac{\sqrt{\varepsilon}}{\tau} \sqrt{E(0)} \\
     \left\|\sqrt{a} \langle \partial_t\rangle^{-1} \left(\eta_t \phi_{t}\right)\right\|_{L^2_x(\S; L^2_t(\R))}&\lesssim \frac{\sqrt{\varepsilon}}{\tau} \sqrt{E(0)}\\
    \left\|\sqrt{a} \langle \partial_t\rangle^{-1} \left(\eta \phi_{t}\right)\right\|_{L^2_x(\S; L^2_t(\R))}&\lesssim \sqrt{\varepsilon} \sqrt{E(0)},
\end{align*}
it suffices to arrive at 
\begin{align}
    \left\|\langle \partial_t\rangle^{-1} \left(\eta(t) \phi_{xx}\right)\right\|_{L^2_x(\S; L^2_t(\R))}&\lesssim \frac{1}{\tau}\sqrt{E(0)}, \label{es:pepxx} \\
     \left\|\langle \partial_t\rangle^{-1} \left(\eta_t(t) \phi_{xx}\right)\right\|_{L^2_x(\S; L^2_t(\R))}&\lesssim \frac{1}{\tau^2}\sqrt{E(0)}. \label{es:petpxx} 
\end{align}

For this purpose we replace $\phi_{xx}$ by the other terms from the wave maps equation to get 
\begin{align*}
   \langle \partial_t\rangle^{-1} \left(\eta(t) \phi_{xx}\right)&= \langle \partial_t\rangle^{-1} \Big(\eta(t) \big( \phi_{tt}+ (|\phi_t|^2- |\phi_x|^2)\phi+ a(x) \phi_t\big)\Big) \\
   &= \langle \partial_t\rangle^{-1} \big(\eta(t) \phi_{tt}\big)+  \langle \partial_t\rangle^{-1} \Big(\eta(t) \left(|\phi_t|^2- |\phi_x|^2\right)\phi\Big)+ a(x) \langle \partial_t\rangle^{-1}\big(\eta(t) \phi_t\big).
\end{align*}
The $L^2_{t, x}$-norms of the first and the last candidates are controlled by 
\begin{gather*}
     \left\|\langle \partial_t\rangle^{-1} \left(\eta(t) \phi_{tt}\right)\right\|_{L^2_x(\S; L^2_t(\R))}\lesssim \frac{1}{\tau} \left\| \phi_{t}\right\|_{L^2_x(\S; L^2_t(-16\pi, 16\pi))}\lesssim \frac{1}{\tau} \sqrt{E(0)},\\
      \left\|a(x) \langle \partial_t\rangle^{-1}\big(\eta(t) \phi_t\big)\right\|_{L^2_x(\S; L^2_t(\R))}\lesssim \left\| \phi_{t}\right\|_{L^2_x(\S; L^2_t(-16\pi, 16\pi))}\lesssim  \sqrt{E(0)}.
\end{gather*}

Concerning the middle term, by ignoring the $\langle \partial_t\rangle^{-1}$ operator for the  $L^2_t$-estimate, it suffices to conclude
\begin{equation*}
     \left\|  \eta(t) \left(|\phi_t|^2- |\phi_x|^2\right)\phi \right\|_{L^2_x(\S; L^2_t(\R))}\lesssim \frac{1}{\tau} \sqrt{E(0)},
\end{equation*}
or, equivalently, 
\begin{equation*}
     \left\|  \eta(t) \left(|\phi_t|^2- |\phi_x|^2\right) \right\|_{L^2_x(\S; L^2_t(\R))}\lesssim \frac{1}{\tau}\sqrt{E(0)}.
\end{equation*}
This is the consequence of  
\begin{align*}
   \left\|  \eta(t) \left(|\phi_t|^2- |\phi_x|^2\right)(t, x) \right\|_{L^2_x(\S; L^2_t(\R))}
     &\leq  \left\|  \left(|\phi_t|^2- |\phi_x|^2\right)(t, x) \right\|_{L^2_x(-\pi, \pi; L^2_t(-16\pi,16\pi))} \\
     &= 4\left\|  \left(\phi_u\cdot \phi_v \right)(t, x) \right\|_{L^2_x(-\pi, \pi; L^2_t(-16\pi,16\pi))} \\
     & \leq 4\left\|  \left(\phi_u\cdot \phi_v \right)(t, x) \right\|_{L^2_{t, x}(D)} \\
     &= 2 \left\|  \left(\phi_u\cdot \phi_v \right)(u, v) \right\|_{L^2_{u, v}(D)} \lesssim E(0)\lesssim \sqrt{E(0)},
\end{align*}
Thus it finishes  the proof of Inequality \eqref{es:pepxx} on  $\langle \partial_t\rangle^{-1} \left(\eta(t) \phi_{xx}\right)$. Similarly, we get the estimate \eqref{es:petpxx}  concerning $\langle \partial_t\rangle^{-1} \left(\eta_t(t) \phi_{xx}\right)$.\\

In conclusion we have obtained that 
\begin{equation}\label{es:sm2}  
\int_{\mathbb{S}^1} a(x) \int_{\mathbb{R}} \Big(\langle \partial_t\rangle^{-1} \left(\eta_{-15\pi}^{15\pi}[\tau](t) \phi_{tx}(t, x)\right)\Big)^2 dt dx\leq C\frac{\sqrt{\varepsilon}}{\tau^2} E(0),
\end{equation}
where the constant $C$ does not depend on $\varepsilon\in (0, 1)$, $\tau\in (0, 1)$ and $E(0)\in [0, 2\pi]$.\\
By combining inequalities \eqref{es:sm1} and \eqref{es:sm2} we can  find some $x_0$ belonging to supp $a$ (or even $[-l_0/2, l_0/2]$) such that 
\begin{equation}
   \int_{-16\pi}^{16\pi} |\phi_t|^2(t, x_0) dt+ \int_{\mathbb{R}} \left(\langle \partial_t\rangle^{-1} \left(\eta_{-15\pi}^{15\pi}[\tau](t) \phi_{tx}(t, x_0)\right)\right)^2 dt\leq C_0 \frac{\sqrt{\varepsilon}}{\tau^2}E(0),
\end{equation}
where the value of $C_0$ is independent of the choice of  $\varepsilon\in (0, 1)$, $\tau\in (0, 1)$, and initial states verifying $E(0)\in [0, 2\pi]$.
\end{proof}

\subsubsection{\bf Propagation of the smallness.}\label{subsec:propsmall}

Now we try to propagate the preceding smallness to arrive at some global bound:
\begin{equation*}
    \|\phi_t\|_{L_x^{\infty}(\S; L^2_t(0, 2\pi))}^2\lesssim \tilde \varepsilon E(0).
\end{equation*}

First we show that this is possible  in a narrow vertical strip with width $|S|$ using bootstrap arguments, then we iterate this procedure to obtain global bounds. Actually, due to the finite speed of propagation instead of vertical strips we shall work with some vertical trapezoidal regions: for $\alpha+ 2\pi<\beta$ and $l\leq \pi$ we  define
\begin{gather}
    P_{\alpha, \beta}^l(y):= \{(t, x): x\in [y, y+l], t\in [\alpha+ x-y, \beta-x+y]\}, 
\end{gather}
and  the $L^{\infty}_x L^2_t$-norm on it as
\begin{gather}
    \|\psi\|_{L^{\infty}_x L^2_t(P_{\alpha, \beta}^l(y))}:= \sup_{x\in [y, y+l]}   \|\psi(t, x)\|_{L^2_t(\alpha+ x-y, \beta-x+y)}.
\end{gather}

The rest part of this subsection is devoted to the proof of the following key lemma.

\begin{lemma}\label{lem:keypropa}
There exist some effectively computable values $S_0$ 
and $C_{S_0}$ such that, for any $\tau\in (0, 1)$,  for any $\alpha\in [-15\pi, 0), \beta\in (2\pi, 15\pi]$, for any $z\in [0, 4\pi]$, and for any initial state of the damped wave maps equation  \eqref{equation:dwm} satisfying 
\begin{gather*}
    E(0)\leq 2\pi,
\end{gather*}
we have
\begin{equation}\label{eq:bound:p181}
    \|\phi_t\|_{L^{\infty}_x L^2_t(P_{\alpha, \beta}^{S_0}(z))}\leq 2  \|\phi_t(t, z)\|_{L^2_t(\alpha, \beta)}+ 6\|\langle\partial_t \rangle^{-1}\left(\eta_{\alpha}^{\beta}[\tau](t) \phi_{tx}(t, z)\right)\|_{L^2_t(\R)},
\end{equation}
and, moreover, by denoting $\tau_0:= S_0/16$, there exists some $\bar x\in (z+ S_0/2, z+ S_0)$ such that 
\begin{align}\label{eq:bound:p182}
   &\;\;\;\;\; \| \langle\partial_t\rangle^{-1} \left(\eta_{\alpha+ S_0+\tau_0}^{\beta-S_0-\tau_0}[\tau_0](t)\phi_{tx}(t, \bar x)\right)\|_{L^2_t(\R)} \notag \\
   &\leq C_{S_0} \left(E(0)\right)^{\frac{1}{4}}\left(  \|\phi_t(t, z)\|_{L^2_t(\alpha, \beta)}+ \|\langle\partial_t \rangle^{-1}\left(\eta_{\alpha}^{\beta}[\tau_0](t) \phi_{tx}(t, z)\right)\|_{L^2_t(\R)} \right)^{\frac{1}{2}}.
\end{align}
\end{lemma}

\begin{proof}[Proof of Lemma~\ref{lem:keypropa}] 
Its proof  is divided in two steps:  first we get a uniform $L^2$ bound for $\phi_t$; then, armed with this uniform $L^2$ bound, by considering the average we are able to find a specific $\bar x$  such that $\phi_{tx}(t, \bar x)$ is bounded in $H^{-1}_t$ space.\\

\noindent {\bf Step 1: On the characterization of $\phi_t$.}

In the following we assume that  $z= x_0\in [0, 4\pi)$. Let given $-15\pi\leq \alpha< 0<2\pi< \beta\leq 15\pi$. Let us consider the damped wave maps equation in the region $P:= P_{\alpha, \beta}^{S}(x_0)$ with the value of $S$ to be chosen later on.  Our goal is to  perform a bootstrap argument to get the desired bound \eqref{eq:bound:p181}.

Recall the equation of $\phi_t$ 
\begin{equation}
    -\phi_{t, uv}= (\phi_u\cdot \phi_v)\phi_t+ (\phi_u\cdot\phi_{t, v}+ \phi_v\cdot \phi_{t, u})\phi+ \frac{1}{4}a \phi_{tt}=: G(t, x), \notag
\end{equation}
 for any $(u, v)\in P$ there is
\begin{equation}
    -\phi_{t, u}(u, v)= -\phi_{t, u}(u, u-2x_0)+ \int_{u- 2x_0}^v G(u, v_0) dv_0, \notag
\end{equation}
and 
\begin{align*}
    -\phi_{t}(u, v)&= -\phi_{t}(2x_0+ v, v)+ \int_{2x_0+v}^u -\phi_{t, u}(u_0, v) du_0,\\
    &= -\phi_{t}(2x_0+ v, v)- \int_{2x_0+v}^u \phi_{t, u}(u_0, u_0- 2x_0) d u_0+ \int_{2x_0+v}^u \int_{u_0- 2x_0}^v G(u_0, v_0) dv_0 d u_0, \\
    &=  -\frac{1}{2}\Big(\phi_t(u- x_0, x_0)+ \phi_t(v+x_0, x_0)+ \phi_x(u- x_0, x_0)- \phi_x(v+ x_0, x_0)\Big)\\
     &\;\;\;\;\;\;+ \int_{2x_0+v}^u \int_{u_0- 2x_0}^v \left( (\phi_u\cdot \phi_v)\phi_t+  \frac{1}{4}a \phi_{tt}\right) (s, y) dv_0 d u_0\\
      &\;\;\;\;\;\;+ \int_{2x_0+v}^u \int_{u_0- 2x_0}^v (\phi_u\cdot\phi_{t, v}+ \phi_v\cdot \phi_{t, u})\phi(u_0, v_0) dv_0 d u_0.
\end{align*}
Simple change of variables yields, 
\begin{equation*}
    \int_{2x_0+v}^u \int_{u_0- 2x_0}^v \psi (s, y) dv_0 d u_0= -2\int_{x_0}^x \int_{t-x+y}^{t+x-y} \psi (s, y) ds dy.
\end{equation*}
Then, thanks to integration by parts, we get 
\begin{align*}
   &\;\;\;\; \int_{2x_0+v}^u \int_{u_0- 2x_0}^v \left((\phi_u\cdot\phi_{t, v})\phi\right)(u_0, v_0) dv_0 d u_0 \\
   &= -\int_{2x_0+v}^u \int_{u_0- 2x_0}^v \left((\phi_u\cdot\phi_{t})\phi_v+ (\phi_{u v}\cdot\phi_{t})\phi\right)(u_0, v_0) dv_0 d u_0 \\
   &\;\;\;\;\;\;+ \int_{2x_0+v}^u (\phi_u\cdot \phi_t)\phi(u_0, v)- (\phi_u\cdot \phi_t)\phi(u_0, u_0- 2x_0)    du_0 \\
   &= \int_{2x_0+v}^u \int_{u_0- 2x_0}^v \left(-(\phi_u\cdot\phi_{t})\phi_v+ \frac{a}{4} |\phi_t|^2 \phi\right)(u_0, v_0) dv_0 d u_0 \\
   &\;\;\;\;\;\;+ \int_{2x_0+v}^u (\phi_u\cdot \phi_t)\phi(u_0, v)- (\phi_u\cdot \phi_t)\phi(u_0, u_0- 2x_0)    du_0 
\end{align*}
and 
\begin{align*}
   &\;\;\;\;\; \int_{2x_0+v}^u \int_{u_0- 2x_0}^v \left((\phi_v\cdot\phi_{t, u})\phi\right)(u_0, v_0) dv_0 d u_0\\
    &= -\int_{v}^{u- 2x_0} \int_{v_0+ 2x_0}^u \left((\phi_v\cdot\phi_{t, u})\phi\right)(u_0, v_0) du_0 d v_0 \\
    &= \int_{v}^{u- 2x_0} \int_{v_0+ 2x_0}^u \left((\phi_v\cdot\phi_{t})\phi_u- \frac{a}{4} |\phi_t|^2 \phi\right)(u_0, v_0) du_0 d v_0 \\
    &\;\;\;\;\;\;\;+ \int_{v}^{u- 2x_0}  (\phi_v \cdot \phi_t)\phi(v_0+ 2x_0, v_0)- (\phi_v \cdot \phi_t)\phi(u, v_0) d v_0\\
    &= \int_{2x_0+v}^u \int_{u_0- 2x_0}^v \left(-(\phi_v\cdot\phi_{t})\phi_u+ \frac{a}{4} |\phi_t|^2 \phi\right)(u_0, v_0) dv_0 d u_0 \\
    &\;\;\;\;\;\;\;+ \int_{v}^{u- 2x_0}  (\phi_v \cdot \phi_t)\phi(v_0+ 2x_0, v_0)- (\phi_v \cdot \phi_t)\phi(u, v_0) d v_0.
\end{align*}

Combining the preceding calculations together we get that for any $(t, x)\in P$ with $x= x_0+d$,
\begin{align*}
    -\phi_t(t, x)&=
    -2\int_{x_0}^x \int_{t-x+y}^{t+x-y} \left( (\phi_u\cdot \phi_v)\phi_t-(\phi_u\cdot\phi_{t})\phi_v-(\phi_v\cdot\phi_{t})\phi_u+ \frac{a}{2} |\phi_t|^2 \phi \right) (s, y) ds dy\\
     &\;\; \;\;\;-\frac{1}{2}\Big(\phi_t(u- x_0, x_0)+ \phi_t(v+x_0, x_0)\Big)\\ 
     &\;\; \;\;\;  -\frac{1}{2} \int_{t- d}^{t+ d} \phi_{tx}(s, x_0) ds\\
    &\;\; \;\;\;-\frac{1}{2} \int_{x_0}^x a(y)\phi_t(t+x-y, y) dy +\frac{1}{2} \int_{x_0}^x a(y)\phi_t(t-x+y, y) dy\\
    &\;\; \;\;\; + \int_{2x_0+v}^u (\phi_u\cdot \phi_t)\phi(u_0, v)du_0- \int_{2x_0+v}^u(\phi_u\cdot \phi_t)\phi(u_0, u_0- 2x_0)    du_0 \\
    &\;\; \;\;\; + \int_{v}^{u- 2x_0}  (\phi_v \cdot \phi_t)\phi(v_0+ 2x_0, v_0)dv_0- \int_{v}^{u- 2x_0} (\phi_v \cdot \phi_t)\phi(u, v_0) d v_0,\\
    &=: \sum_{j=1}^6 f_j(t, x).
\end{align*}
 For some fixed $x=x_0+ d$ with $d\leq|S|$, by the construction of $P^S_{\alpha, \beta}(x_0)$ we are interested in  $t\in [\alpha+d, \beta-d]$: in the following we shall estimate the values of $|f_j(t, x)|$ or $\|f_j(t, x)\|_{L^2_t(\alpha+d, \beta-d)}$ for $j\in \{1,2..., 6\}$ successively.\\

\noindent {\bf 1)} Concerning $f_2(t, x)$, since $x= x_0+ d$, there is
\begin{equation*}
    \phi_t(u- x_0, x_0)= \phi_t(t+ d, x_0), \; \phi_t(v+ x_0, x_0)= \phi_t(t- d, x_0),
\end{equation*}
which immediately leads to 
\begin{equation}\label{es:f2}
    \|f_2(t, x)\|_{L^2_t(\alpha+d, \beta-d)}\leq \|\phi_t(t, x_0)\|_{L^2_t(\alpha, \beta)}.
\end{equation}

\noindent {\bf 2)} Concerning $f_4$ we know that
\begin{align*}
    \int_{\alpha+d}^{\beta-d}\left(\int_{x_0}^x a(y)\phi_t(t+x-y, y) dy\right)^2 dt&\lesssim  d \int_{\alpha+d}^{\beta-d}\int_{x_0}^x \phi_t^2(t+x-y, y) dy dt\\
    &\lesssim d^2 \|\phi_t\|_{L^{\infty}_x L^2_t(P)}^2,
\end{align*}
thus
\begin{equation}\label{es:f4}
     \|f_4(t, x)\|_{L^2_t(\alpha+d, \beta-d)}\lesssim d \|\phi_t\|_{L^{\infty}_x L^2_t(P)}.
\end{equation}

\noindent {\bf 3)} Now we turn to $f_5$ and $f_6$. By noticing  
\begin{align*}
   \left( \int_{2x_0+v}^u (\phi_u\cdot \phi_t)\phi(u_0, v) du_0\right)^2&\leq \int_{2x_0+v}^u |\phi_u|^2(u_0, v) du_0 \int_{2x_0+v}^u |\phi_t|^2(u_0, v) du_0\\
   &\leq \|\phi_u\|_{L^2_u L^{\infty}_v(P)}^2 \int_{x_0}^{x_0+ d} |\phi_t|^2(t-x_0-d+y, y)dy
\end{align*}
we get
\begin{align*}
    \int_{\alpha+d}^{\beta-d}\left( \int_{2x_0+v}^u (\phi_u\cdot \phi_t)\phi(u_0, v) du_0\right)^2 dt&\lesssim E(0)   \int_{\alpha+d}^{\beta-d}\int_{x_0}^{x_0+ d} |\phi_t|^2(t-x_0-d+y, y)dy dt \\
    &\lesssim d E(0) \|\phi_t\|_{L^{\infty}_x L^2_t(P)}^2.
\end{align*}
Similarly,
\begin{align*}
 \int_{2x_0+v}^u (\phi_u\cdot \phi_t)\phi(u_0, u_0- 2x_0) du_0&=   \int_{t-d}^{t+d} (\phi_u\cdot \phi_t)\phi(s, x_0) ds \\
   &= \frac{1}{2} \int_{t-d}^{t+d} (\phi_x\cdot \phi_t)\phi(s, x_0) ds +  \frac{1}{2} \int_{t-d}^{t+d} |\phi_t|^2\phi(s, x_0) ds. 
\end{align*}
 Noticing that  
\begin{gather*}
    \int_{t-d}^{t+d} |\phi_t|^2(s, x_0) ds\leq \|\phi_t(t, x_0)\|_{L^2_t(\alpha, \beta)}^2,\\
   \int_{t-d}^{t+d} |\phi_x|^2(s, x_0) ds\leq \|\phi_x(t, x_0)\|_{L^2_t(\alpha, \beta)}^2\lesssim E(0),   
\end{gather*}
where in the last inequality we have applied the energy estimate \eqref{es:e3}, 
we have 
\begin{align*}
   \int_{\alpha+d}^{\beta-d}  \left(\int_{t-d}^{t+d} |\phi_t|^2\phi(s, x_0) ds\right)^2 dt&\leq \|\phi_t(t, x_0)\|_{L^2_t(\alpha, \beta)}^2 \int_{\alpha+d}^{\beta-d} \int_{t-d}^{t+d} |\phi_t|^2(s, x_0) ds dt\\
   &\leq 2d \|\phi_t(t, x_0)\|_{L^2_t(\alpha, \beta)}^4,
\end{align*}
and 
\begin{align*}
   \int_{\alpha+d}^{\beta-d}  \left(\int_{t-d}^{t+d} (\phi_x\cdot \phi_t)\phi(s, x_0) ds\right)^2 dt&\leq  \int_{\alpha+d}^{\beta-d}  \left(\int_{t-d}^{t+d} \phi_x^2(s, x_0) ds\right)  \left(\int_{t-d}^{t+d} \phi_t^2(s, x_0) ds\right) dt \\
   &\leq 2d  \|\phi_x(t, x_0)\|_{L^2_t(\alpha, \beta)}^2  \|\phi_t(t, x_0)\|_{L^2_t(\alpha, \beta)}^2 \\
   &\lesssim d  E(0)  \|\phi_t(t, x_0)\|_{L^2_t(\alpha, \beta)}^2.
\end{align*}
Hence  
\begin{align}     
 \|f_5(t, x)\|_{L^2_t(\alpha+d, \beta-d)}+ \|f_6(t, x)\|_{L^2_t(\alpha+d, \beta-d)} 
&\lesssim \sqrt{d} \sqrt{E(0)} \|\phi_t\|_{L^{\infty}_x L^2_t(P)}+ \sqrt{d} \|\phi_t(t, x_0)\|_{L^2_t(\alpha, \beta)}^2 \notag\\
&\lesssim \sqrt{d} \sqrt{E(0)} \|\phi_t\|_{L^{\infty}_x L^2_t(P)}\label{es:f56}
\end{align} 

\noindent {\bf 4)}  Then, we turn to $f_1$, again, by treating the items one by one. Since
\begin{align*}
\int_{x_0}^x \int_{t-x+y}^{t+x-y} \left( (\phi_u\cdot \phi_v)\phi_t\right) (s, y) ds dy&\leq \|\phi_u\cdot \phi_v\|_{L^2_{t,x}}\|\phi_t\|_{L^2_{t,x}} \\
&\leq 2\sqrt{d}\|\phi_u\cdot \phi_v\|_{L^2_{u, v}}\|\phi_t\|_{L^{\infty}_x L^2_t}\\
&\lesssim \sqrt{d}E(0) \|\phi_t\|_{L^{\infty}_x L^2_t(P)}
\end{align*}
and
\begin{align*}
  \int_{x_0}^x \int_{t-x+y}^{t+x-y} \left( \frac{a}{2} |\phi_t|^2 \phi\right) (s, y) ds dy\lesssim d    \|\phi_t\|_{L^{\infty}_x L^2_t(P)}^2,
\end{align*}
we have
\begin{align*}
    |f_1|(t, x)&\lesssim \sqrt{d}E(0) \|\phi_t\|_{L^{\infty}_x L^2_t(P)}+ d    \|\phi_t\|_{L^{\infty}_x L^2_t(P)}^2 \\
    &\lesssim \left(\sqrt{d}E(0)+ d\sqrt{E(0)}\right)\|\phi_t\|_{L^{\infty}_x L^2_t(P)},
\end{align*}
which further implies that
\begin{equation}\label{es:f1}
    \|f_1(t, x)\|_{L^2_t(\alpha+d, \beta-d)}\lesssim \left(\sqrt{d}E(0)+ d\sqrt{E(0)}\right)\|\phi_t\|_{L^{\infty}_x L^2_t(P)}.
\end{equation}

\noindent {\bf 5)} Finally, we deal with the $\|f_3(t, x)\|_{L^2_t(\alpha+d, \beta-d)}$ term,  for which one also requires the smallness of $\phi_{tx}$ to close the loop. Indeed, it sounds natural to replace the integration of $\phi_{tx}$ by  $\phi_x$ to simplify the presentation,
\begin{equation*}
    \int_{t- d}^{t+ d} \phi_{tx}(s, x_0) ds= \phi_x(t+d, x_0)- \phi_x(t-d, x_0),
\end{equation*}
However, by performing this integration we immediately lose  the control of smallness: because we are not able to compensate the  $\|\phi_x(t, x_0)\|_{L^2_t(\alpha, \beta)}^2$ term in our bootstrap argument, which is not supposed to be small in general.  

Let us define 
\begin{equation*}
g(t):=     \eta_{\alpha}^{\beta}[\tau](t)  \phi_{tx}(t, x_0) \;  \textrm{ with } \tau\in (0, 1),
\end{equation*}
for which we assume that the upper bound of $\|\langle \partial_t\rangle^{-1} g(t)\|_{L^2_t(\R)}$ is known.
Then,  for $t\in (\alpha+d, \beta-d)$ and $s\in (t-d, t+d)$ there is 
\begin{equation*}
    \phi_{tx}(s, x_0)=  \eta_{\alpha}^{\beta}[\tau](s)  \phi_{tx}(s, x_0)= g(s).
\end{equation*}
This implies that for $x= x_0+d$ and $t\in (\alpha+d, \beta-d)$,
\begin{equation*}
    f_3(t, x)= -\frac{1}{2}\int_{t-d}^{t+d} g(s) ds.
\end{equation*}
By performing the inverse Fourier transformation we get 
\begin{equation*}
    g(s)= \int_{\lambda\in\R} e^{is\lambda} \hat{g}(\lambda) d\lambda,
\end{equation*}
thus 
\begin{align*}
    -2  f_3(t, x)&= \int_{t-d}^{t+d} \int_{\lambda\in\R} e^{is\lambda} \hat{g}(\lambda) d\lambda ds \\
    &= \int_{t-d}^{t+d} \int_{|\lambda|\leq 1} e^{is\lambda} \hat{g}(\lambda) d\lambda ds+ \int_{t-d}^{t+d} \int_{|\lambda|> 1} e^{is\lambda} \hat{g}(\lambda) d\lambda ds.
\end{align*}

Concerning the first candidate in the preceding formula there is a trivial bound
\begin{equation*}
    \left|\left(\int_{|\lambda|\leq 1} e^{is\lambda} \hat{g}(\lambda) d\lambda\right)\right|^2\leq  2\int_{|\lambda|\leq 1}  \langle\lambda\rangle^{-2}|\hat{g}(\lambda)|^2 d\lambda\leq 2\|\langle \partial_t\rangle^{-1} g(t)\|_{L^2_t(\R)}^2.
\end{equation*}
Therefore,
\begin{equation}
    \left|\int_{t-d}^{t+d} \int_{|\lambda|\leq 1} e^{is\lambda} \hat{g}(\lambda) d\lambda ds\right|\leq 3d\|\langle \partial_t\rangle^{-1} g(t)\|_{L^2_t(\R)}. \notag
\end{equation}

Next, we turn to the high frequency part of $f_3(t, x)$: thanks to symmetry reasons it further suffices to treat the positive high frequency part,
\begin{align*}
    \int_{t-d}^{t+d} \int_{\lambda> 1} e^{is\lambda} \hat{g}(\lambda) d\lambda ds&=  \int_{\lambda> 1} \hat{g}(\lambda) \int_{t-d}^{t+d} e^{is\lambda}  ds d\lambda \\
    &= -i \int_{\lambda> 1} \hat{g}(\lambda) \left(\frac{e^{i(t+d)\lambda}}{\lambda}- \frac{e^{i(t-d)\lambda}}{\lambda}\right) d\lambda.
\end{align*}
By denoting 
\begin{equation*}
    F(\bar u):= \int_{\R} \left(\frac{\chi_{\lambda>1}}{\lambda} \hat{g}(\lambda)\right) e^{i \lambda \bar u} d\lambda
\end{equation*}
we know that 
\begin{align*}
    \|\int_{\lambda> 1} \hat{g}(\lambda) \frac{e^{i(t+d)\lambda}}{\lambda} d\lambda\|_{L^2_t(\alpha+d, \beta-d)}&=  \|\int_{\lambda> 1} \hat{g}(\lambda) \frac{e^{i\bar u\lambda}}{\lambda} d\lambda\|_{L^2_{\bar u}(\alpha+2d, \beta)}\\
    &= \|F(\bar u)\|_{L^2_{\bar u}(\alpha+2d, \beta)}\\
    &\leq \|F(\bar u)\|_{L^2_{\bar u}(\R)}\\
    &= \|\frac{\chi_{\lambda>1}}{\lambda} \hat{g}(\lambda)
    \|_{L^2_{\lambda}(\R)}\\
    &\leq \frac{3}{2}\|\langle \partial_t\rangle^{-1} g(t)
    \|_{L^2_{t}(\R)}.
\end{align*}
Similarly
\begin{align*}
    \|\int_{\lambda> 1} \hat{g}(\lambda) \frac{e^{i(t-d)\lambda}}{\lambda} d\lambda\|_{L^2_t(\alpha+d, \beta-d)}\leq \frac{3}{2}\|\langle \partial_t\rangle^{-1} g(t)
    \|_{L^2_{t}(\R)}.
\end{align*}
Hence
\begin{equation}\label{es:f3}
    \|f_3(t, x)\|_{L^2_t(\alpha+d, \beta-d)}\leq \left(\frac{3}{2}+ 15d\right)\|\langle \partial_t\rangle^{-1} \left(\eta_{\alpha}^{\beta}[\tau](t) \phi_{tx}(t, x_0)\right)
    \|_{L^2_{t}(\R)}.
\end{equation}

In conclusion, by combining the estimates \eqref{es:f2}, \eqref{es:f4}, \eqref{es:f56}, \eqref{es:f1} and \eqref{es:f3}  we obtain: for $x= x_0+ d$ and for any $\tau\in (0, 1)$, there is 
\begin{align}
   &\;\;\;\; \|\phi_t(t, x)\|_{L^2_t(\alpha+d, \beta-d)} \notag \\
    &\leq \|\phi_t(t, x_0)\|_{L^2_t(\alpha, \beta)}+ 3\|\langle \partial_t\rangle^{-1} \left(\eta_{\alpha}^{\beta}[\tau](t) \phi_{tx}(t, x_0)\right)
    \|_{L^2_{t}(\R)} + C\left(d+ \sqrt{d} \sqrt{E(0)}\right)\|\phi_t\|_{L^{\infty}_x L^2_t(P)} \notag\\
    &\leq \|\phi_t(t, x_0)\|_{L^2_t(\alpha, \beta)}+ 3\|\langle \partial_t\rangle^{-1} \left(\eta_{\alpha}^{\beta}[\tau](t) \phi_{tx}(t, x_0)\right)
    \|_{L^2_{t}(\R)}  + \widetilde C \sqrt{d} \|\phi_t\|_{L^{\infty}_x L^2_t(P)}, \notag
\end{align}
where the constant $\widetilde C$ is independent of $\alpha\in [-15\pi, 0), \beta\in (2\pi, 15\pi], d\in (0, 1/10), \tau\in (0, 1), x_0\in [0, 4\pi]$ and $E(0)\leq 2\pi$. 
In order to close the loop of bootstrap, it suffices to define 
\begin{equation}\label{def:S0}
    S_0:= \left(\frac{1}{2 \widetilde C}\right)^2.
\end{equation}
The preceding inequality immediately yields 
\begin{equation*}
    \|\phi_t\|_{L^{\infty}_x L^2_t(P^{S_0}_{\alpha, \beta}(x_0))}\leq 2\|\phi_t(t, x_0)\|_{L^2_t(\alpha, \beta)}+ 6\|\langle \partial_t\rangle^{-1} \left(\eta_{\alpha}^{\beta}[\tau](t) \phi_{tx}(t, x_0)\right)
    \|_{L^2_{t}(\R)}. 
\end{equation*}

\noindent {\bf Step 2: on the choice of $z$ such that $\phi_{tx}(t, z)$ is small.}

Different from  $\phi_t$ for which we have get a uniform bound on $L^2_t$-norm on a strip with width $S_0$, for the  $\phi_{tx}$ term  we only select one specific slide, namely some  $\bar x\in [x_0+ S_0/2, x_0+ S_0]$ such that the $H^{-1}_t$-norm of $\phi_{tx}(\cdot, \bar x)$ is small. 
First we construct a $2\pi$-periodic, non-negative, smooth cutoff function $b(x)$,
\begin{equation}
    b(x)=1 \textrm{ on } \left[\frac{S_0}{8},  \frac{3S_0}{8}\right], \;\;  b(x)=0 \textrm{ on } \S\setminus\left[0, \frac{S_0}{2}\right],
\end{equation}
and further define $b_{z}$ as 
\begin{equation}
    b_z(x):= b(x-z).
\end{equation}

Armed with the smallness of 
\begin{equation*}
    \|\phi_t(t, x)\|_{L^{\infty}_x(x_0+S_0/2, x_0+ S_0; L^2_t(\alpha+ S_0, \beta- S_0))}
\end{equation*}
being proved in {\bf Step 1},
 in this step we  use it  to dominate the following:
\begin{equation*}  
\int_{\mathbb{S}^1} b_{x_0+ S_0/2}(x) \int_{\mathbb{R}} \left(\langle \partial_t\rangle^{-1} \left(\eta_{\alpha+ S_0+ \tau_0}^{\beta-S_0-\tau_0}[\tau_0](t) \phi_{tx}(t, x)\right)\right)^2 dt dx.
\end{equation*}
Since $\tau_0$ and $S_0$ are fixed, the $C^2$-norms of the truncated functions $\eta_{\alpha+ S_0+ \tau_0}^{\beta- S_0-\tau_0}[\tau_0](t)$ and $b_z(x)$ are uniformly bounded. For  ease of notations, in the rest part of this section we simply denote the functions $\eta_{\alpha+ S_0+ \tau_0}^{\beta- S_0-\tau_0}[\tau_0](t)$ by $\eta(t)$ and $b_{x_0+ S_0/2}(x)$ by $\tilde b(x)$.\\

Similar to the proof of Lemma~\ref{lemma:choose}, there is 
\begin{gather*}
   \|\langle \partial_t\rangle^{-1} \left(\eta(t) \phi_{tx}(t, x)\right)\|_{L^2_t(\R)}\lesssim  \|\phi_x\|_{L^2_t(\alpha+ S_0, \beta- S_0)}\\
    \|\langle \partial_t\rangle^{-1} \left(\eta(t) \phi_{tt}(t, x)\right)\|_{L^2_t(\R)}\lesssim  \|\phi_t\|_{L^2_t(\alpha+ S_0, \beta- S_0)}, \\
   \|\langle \partial_t\rangle^{-1} \left(\eta(t) \phi_{tx}(t, x)\right)\|_{L^2_{t, x}(\R\times \S)}^2\lesssim  E(0),
\end{gather*}
and
\begin{align*}
&\;\;\;\; \int_{\mathbb{S}^1}  \int_{\mathbb{R}}\tilde b(x) \left(\langle \partial_t\rangle^{-1} \left(\eta(t) \phi_{tx}(t, x)\right)\right)\cdot  \left(\langle \partial_t\rangle^{-1} \left(\eta(t) \phi_{tx}(t, x)\right)\right) dt dx \\
&=  \int_{\mathbb{S}^1}  \int_{\mathbb{R}}\tilde b(x) \left(\langle \partial_t\rangle^{-1} \left(\eta(t) \phi_{tt}\right)\right)\cdot  \left(\langle \partial_t\rangle^{-1} \left(\eta(t) \phi_{xx}\right)\right) dt dx \\
&\;\;\;+ \int_{\mathbb{S}^1}  \int_{\mathbb{R}}\tilde b(x) \left(\langle \partial_t\rangle^{-1} \left(\eta_t(t) \phi_{t}\right)\right)\cdot  \left(\langle \partial_t\rangle^{-1} \left(\eta(t) \phi_{xx}\right)\right) dt dx \\
&\;\;\;+ \int_{\mathbb{S}^1}  \int_{\mathbb{R}}\tilde b(x) \left(\langle \partial_t\rangle^{-1} \left(\eta(t) \phi_{t}\right)\right)\cdot  \left(\langle \partial_t\rangle^{-1} \left(\eta_t(t) \phi_{xx}\right)\right) dt dx \\
&\;\;\;- \int_{\mathbb{S}^1}  \int_{\mathbb{R}}\tilde b_x(x) \left(\langle \partial_t\rangle^{-1} \left(\eta_t(t) \phi_{t}\right)\right)\cdot  \left(\langle \partial_t\rangle^{-1} \left(\eta(t) \phi_{tx}\right)\right) dt dx \\
&=: I+ II+ III+ IV.
\end{align*}

The same calculation yields
\begin{align*}
    IV &\lesssim \sqrt{E(0)} \left\|\sqrt{\tilde b(x)}  \left(\eta_t(t) \phi_{t}\right)\right\|_{L^2_x(\S; L^2_t(\R))}   \\
    &\lesssim  \sqrt{E(0)}\|\phi_t(t, x)\|_{L^{\infty}_x(x_0+S_0/2, x_0+ S_0; L^2_t(\alpha+ S_0, \beta- S_0))}.
\end{align*}
We also know that 
\begin{align*}
    \left\|\sqrt{\tilde b} \langle \partial_t\rangle^{-1} \left(\eta \phi_{tt}\right)\right\|_{L^2_x(\S; L^2_t(\R))}&\lesssim  \left\|  \phi_{t}\right\|_{L^2_x(x_0+ S_0/2, x_0+ S_0; L^2_t(\textrm{supp }\eta))} \\
     &\lesssim  \|\phi_t(t, x)\|_{L^{\infty}_x(x_0+S_0/2, x_0+ S_0; L^2_t(\alpha+ S_0, \beta- S_0))}
\end{align*}
and  
\begin{align*}
     \left\|\sqrt{\tilde b} \langle \partial_t\rangle^{-1} \left(\eta_t \phi_{t}\right)\right\|_{L^2_x(\S; L^2_t(\R))}&\lesssim  \|\phi_t(t, x)\|_{L^{\infty}_x(x_0+S_0/2, x_0+ S_0; L^2_t(\alpha+ S_0, \beta- S_0))} \\
    \left\|\sqrt{\tilde b} \langle \partial_t\rangle^{-1} \left(\eta \phi_{t}\right)\right\|_{L^2_x(\S; L^2_t(\R))}&\lesssim \|\phi_t(t, x)\|_{L^{\infty}_x(x_0+S_0/2, x_0+ S_0; L^2_t(\alpha+ S_0, \beta- S_0))}. 
\end{align*}
Again, by adapting the wave maps equation we obtain 
\begin{equation*}
    \left\|\langle \partial_t\rangle^{-1} \left(\eta(t) \phi_{xx}\right)\right\|_{L^2_x(\S; L^2_t(\R))}+  \left\|\langle \partial_t\rangle^{-1} \left(\eta_t(t) \phi_{xx}\right)\right\|_{L^2_x(\S; L^2_t(\R))}\lesssim \sqrt{E(0)}.
\end{equation*}
Therefore 
\begin{equation*}
    I+ II+ III\lesssim \sqrt{E(0)} \|\phi_t(t, x)\|_{L^{\infty}_x(x_0+S_0/2, x_0+ S_0; L^2_t(\alpha+ S_0, \beta- S_0))},
\end{equation*}
which, together with the estimate on $IV$, yields
\begin{align}  
&\;\;\;\;\;  \int_{\mathbb{S}^1} b_{x_0+ S_0/2}(x) \int_{\mathbb{R}} \left(\langle \partial_t\rangle^{-1} \left(\eta_{\alpha+ S_0+ \tau_0}^{\beta-S_0-\tau_0}[\tau_0](t) \phi_{tx}(t, x)\right)\right)^2 dt dx \notag\\
&\leq C_H  \sqrt{E(0)}\|\phi_t(t, x)\|_{L^{\infty}_x(x_0+S_0/2, x_0+ S_0; L^2_t(\alpha+ S_0, \beta- S_0))}, \notag
\end{align}
where the value of $C_H$ does not depend on the choice of $x_0\in \S, \alpha\in [-15\pi, 0], \beta\in [2\pi, 15\pi]$ and $E(0)\leq 2\pi$. \\

As a direct consequence of the preceding inequality, we are able to find  some point $\bar x\in [x_0+ \frac{5 S_0}{8}, x_0+ \frac{7 S_0}{8}]$ such that 
\begin{align*}
&\;\;\;\;\;  \int_{\mathbb{R}} \left(\langle \partial_t\rangle^{-1} \left(\eta_{\alpha+ S_0+ \tau_0}^{\beta-S_0-\tau_0}[\tau_0](t) \phi_{tx}(t, \bar x)\right)\right)^2 dt \\
&\leq \frac{4 C_H}{S_0}\sqrt{E(0)}\|\phi_t(t, x)\|_{L^{\infty}_x(x_0+S_0/2, x_0+ S_0; L^2_t(\alpha+ S_0, \beta- S_0))} \\
&\leq \frac{24 C_H}{S_0}\sqrt{E(0)} \left(  \|\phi_t(t, x_0)\|_{L^2_t(\alpha, \beta)}+ \|\langle\partial_t \rangle^{-1}\left(\eta_{\alpha}^{\beta}[\tau_0](t) \phi_{tx}(t, x_0)\right)\|_{L^2_t(\R)} \right).
\end{align*}
This concludes the proof of Inequality \eqref{eq:bound:p182} by choosing 
\begin{equation}\label{def:CS0}
    C_{S_0}:= \left(\frac{24 C_H}{S_0}\right)^{\frac{1}{2}}.
\end{equation}

\end{proof}

\subsubsection{\bf Proof of Proposition~\ref{prop:2}}\label{subsec:propro2}
Armed with Lemma~\ref{lemma:choose} and Lemma~\ref{lem:keypropa} we show that Proposition~\ref{prop:2} is a direct consequence of these properties.

\begin{proof}[Proof of Proposition~\ref{prop:2}]
 We start the proof by fixing some constants: $C_0$ by Lemma~\ref{lemma:choose},  $S_0$  by \eqref{def:S0}, $C_{S_0}$ by \eqref{def:CS0} and $\tau_0$ as $S_0/16$.
 Let $\phi$ be a solution of the damped wave maps equation \eqref{equation:dwm} satisfying $E(0)\leq 2\pi$.  Assume as we may that 
\begin{equation*}
      \int_{-16\pi}^{16\pi}\int_{\mathbb{S}^1}a(x)|\phi_t|^2(t, x) dxdt\leq \varepsilon  E(0).
\end{equation*}

\noindent {\it Step 0:} Define $(\alpha_0, \beta_0):= (-15\pi, 15\pi)$. By applying Lemma~\ref{lemma:choose} we are able to find some $x_0\in [0, 2\pi)$ such that 
\begin{equation*}
    \max\left\{ \|\phi_t(t, x_0)\|_{L^2_t(\alpha_0, \beta_0)}, \;    \|\langle\partial_t\rangle^{-1}\left(\eta_{\alpha_0}^{\beta_0}[\tau_0] \phi_{tx}\right)(t, x_0)\|_{L^2_t(\R)}\right\} \leq \varepsilon^{\frac{1}{4}} \left(\frac{C_0}{\tau_0^2}\right)^{\frac{1}{2}} (E(0))^{\frac{1}{2}}.
\end{equation*}

\noindent {\it Step 1:}  Next, we define $(\alpha_1, \beta_1):= (\alpha_0+ S_0+ \tau_0, \beta_0- S_0- \tau_0)$. Since $\tau_0\in (0, 1), \alpha_0\in [-15\pi, 0), \beta_0\in (2\pi, 15\pi]$ and $x_0\in [0, 2\pi]$,  we are allowed to apply Lemma~\ref{lem:keypropa} to find some point  $x_1$ belongs to $[x_0+ S_0/2, x_0+ S_0]$ such that
\begin{align*}
     & \;\;\;\;\;  \|\phi_t\|_{L^{\infty}_x(x_0, x_1; L^2_t(\alpha_1, \beta_1))} \\
     &\leq 2  \|\phi_t(t, x_0)\|_{L^2_t(\alpha_0, \beta_0)}+ 6\|\langle\partial_t \rangle^{-1}\left(\eta_{\alpha_0}^{\beta_0}[\tau_0](t) \phi_{tx}(t, x_0)\right)\|_{L^2_t(\R)}, \\
     &\leq 8 \varepsilon^{\frac{1}{4}} \left(\frac{C_0}{\tau_0^2}\right)^{\frac{1}{2}} (E(0))^{\frac{1}{2}},
\end{align*}
and 
\begin{align}
   &\;\;\;\;\; \|\langle \partial_t\rangle^{-1} \left(\eta_{\alpha_1}^{\beta_1}[\tau_0](t)\phi_{tx}(t, x_1)\right)\|_{L^2_t(\R)} \notag \\
   &\leq C_{S_0} \left(E(0)\right)^{\frac{1}{4}}\left(  \|\phi_t(t, x_0)\|_{L^2_t(\alpha_0, \beta_0)}+ \|\langle\partial_t \rangle^{-1}\left(\eta_{\alpha_0}^{\beta_0}[\tau_0](t) \phi_{tx}(t, x_0)\right)\|_{L^2_t(\R)} \right)^{\frac{1}{2}} \notag  \\
   &\leq 2 C_{S_0}\varepsilon^{\frac{1}{8}} \left(\frac{C_0}{\tau_0^2}\right)^{\frac{1}{4}} (E(0))^{\frac{1}{2}}. \notag
\end{align}
Thus 
\begin{align*}
     \max\left\{\|\phi_t\|_{L^{\infty}_x(x_0, x_1; L^2_t(\alpha_1, \beta_1))}, \;  \| \langle \partial_t\rangle^{-1} \left(\eta_{\alpha_1}^{\beta_1}[\tau_0](t)\phi_{tx}(t, x_1)\right)\|_{L^2_t(\R)}\right\}\leq  2 C_{S_0}\varepsilon^{\frac{1}{8}} \left(\frac{C_0}{\tau_0^2}\right)^{\frac{1}{2}} (E(0))^{\frac{1}{2}}.
\end{align*}

\noindent {\it Step 2:} Then, we define $(\alpha_2, \beta_2):= (\alpha_0+ 2S_0+ 2\tau_0, \beta_0- 2S_0- 2\tau_0)$. Since $\tau_0\in (0, 1), \alpha_1\in [-15\pi, 0), \beta_1\in (2\pi, 15\pi]$ and $x_1\in [0, 4\pi]$,  we are allowed to apply Lemma~\ref{lem:keypropa} to find some point  $x_2$ belongs to $[x_1+ S_0/2, x_1+ S_0]$ such that
\begin{align*}
    \max\left\{\|\phi_t\|_{L^{\infty}_x(x_1, x_2; L^2_t(\alpha_2, \beta_2))}, \;  \|\langle \partial_t\rangle^{-1} \left(\eta_{\alpha_2}^{\beta_2}[\tau_0](t)\phi_{tx}(t, x_2)\right)\|_{L^2_t(\R)}\right\} \leq (2 C_{S_0})^2\varepsilon^{\frac{1}{16}} \left(\frac{C_0}{\tau_0^2}\right)^{\frac{1}{2}} (E(0))^{\frac{1}{2}}.
\end{align*}

\noindent {\it Step 3:} We iterate this procedure: suppose that for every $k\in \{1, 2,..., n\}$ we have found 
\begin{equation*}
    (\alpha_k, \beta_k)= (\alpha_0+ kS_0+ k\tau_0, \beta_0- kS_0-k\tau_0) \textrm{ and } x_k\in [x_{k-1}+ S_0/2, x_{k-1}+ S_0], 
\end{equation*}
such that 
\begin{align*}
     \max\left\{\|\phi_t\|_{L^{\infty}_x(x_{k-1}, x_k; L^2_t(\alpha_k, \beta_k))}, \;  \|\langle \partial_t\rangle^{-1} \left(\eta_{\alpha_k}^{\beta_k}[\tau_0](t)\phi_{tx}(t, x_k)\right)\|_{L^2_t(\R)}\right\} \leq (2 C_{S_0})^k\varepsilon^{\frac{1}{2^{k+2}}} \left(\frac{C_0}{\tau_0^2}\right)^{\frac{1}{2}} (E(0))^{\frac{1}{2}}.
\end{align*}

If $x_{n-1}<x_0+ 2\pi\leq x_n$, then we stop the procedure. Now we notice from the choices of $x_k$ that $n\in [\frac{2\pi}{S_0}, \frac{4\pi}{S_0}]$, and that for every $k\in \{1, 2,..., n\}$ we have $\alpha_k\in [-15\pi, 0), \beta_k\in [3\pi, 15\pi]$ and $x_k\in [0, 4\pi]$.

If $x_{n}<x_0+ 2\pi$, then since $n\leq \frac{4\pi}{S_0}$ we conclude from the choice of $(\alpha_k, \beta_k, x_k)$ that $\alpha_n\in [-15\pi, 0), \beta_n\in (2\pi, 15\pi]$ and $x_n\in [0, 4\pi]$. Consequently, we are allowed to use Lemma~\ref{lem:keypropa} to find some point $x_{n+1}\in [x_n+ S_0/2, x_n+ S_0]$ such that 
\begin{align*}
     &\;\;\;\;\;  \max\left\{\|\phi_t\|_{L^{\infty}_x(x_{n}, x_{n+1}; L^2_t(\alpha_{n+1}, \beta_{n+1}))}, \;  \|\langle \partial_t\rangle^{-1} \left(\eta_{\alpha_{n+1}}^{\beta_{n+1}}[\tau_0](t)\phi_{tx}(t, x_{n+1})\right)\|_{L^2_t(\R)}\right\} \\
     &\leq (2 C_{S_0})^{n+1}\varepsilon^{\frac{1}{2^{n+3}}} \left(\frac{C_0}{\tau_0^2}\right)^{\frac{1}{2}} (E(0))^{\frac{1}{2}},
\end{align*}
where $ (\alpha_{n+1}, \beta_{n+1})= (\alpha_0+ (n+1)S_0+ (n+1)\tau_0, \beta_0- (n+1)S_0- (n+1)\tau_0)$. \\

\noindent {\it Step 4:} In conclusion we have found some $N\in [\frac{2\pi}{S_0}, \frac{4\pi}{S_0}]$ such that $x_{N-1}< x_0+ 2\pi\leq x_N$ and that for every $k\in \{1, 2,..., N\}$ there is
\begin{gather*}
    \alpha_k= \alpha+0+ kS_0+ k\tau_0\in [-15\pi, 0), \\
    \beta_k= \beta_0- k S_0- k\tau_0\in [3\pi, 15\pi], \\
    x_k\in [x_{k-1}+ S_0/2, x_{k-1}+ S_0],
\end{gather*}
and 
\begin{align*}
     &\;\;\;\;\;  \max\left\{\|\phi_t\|_{L^{\infty}_x(x_{k-1}, x_k; L^2_t(\alpha_k, \beta_k))}, \;  \|\langle \partial_t\rangle^{-1} \left(\eta_{\alpha_k}^{\beta_k}[\tau_0](t)\phi_{tx}(t, x_k)\right)\|_{L^2_t(\R)}\right\} \\
     &\leq (2C_{S_0})^k\varepsilon^{\frac{1}{2^{k+2}}} \left(\frac{C_0}{\tau_0^2}\right)^{\frac{1}{2}} (E(0))^{\frac{1}{2}}.
\end{align*}
Hence, thanks to the $2\pi$-periodicity of $\phi(x)$, we know that 
\begin{equation*}
    \|\phi_t\|_{L^{\infty}_x(0, 2\pi; L^2_t(0, 3\pi))}\leq (2C_{S_0})^N\varepsilon^{\frac{1}{2^{N+2}}} \left(\frac{C_0}{\tau_0^2}\right)^{\frac{1}{2}} (E(0))^{\frac{1}{2}},
\end{equation*}
which concludes the proof of Proposition~\ref{prop:2} by setting $\varepsilon_0= \varepsilon_0(\delta)$ in such fashion that
\begin{equation*}
  (2C_{S_0})^{2 [4\pi/S_0]}\varepsilon_0^{\frac{1}{2^{[4\pi/S_0]+1}}} \left(\frac{C_0}{\tau_0^2}\right)= \delta.
\end{equation*}
In other words there exist some $p$ and $C_p$ effectively computable such that  in Proposition~\ref{prop:2},
\begin{equation}
    \varepsilon_0=\varepsilon_0(\delta) = C_p \delta^{p}.
\end{equation}

\end{proof}

\subsection{Proof of Proposition~\ref{prop:3}}\label{subsec:prop3}
This section is devoted to the proof of the local exponential stability of the damped wave maps equation: the aim is to  benefit from the exponential stability of the scalar damped wave equations and treat the nonlinear terms as perturbation.

 Recall the following quantitative results concerning  stabilization of wave equations, the explicit decay rate for the one dimensional damped wave equation can be calculated directly. We also refer to \cite{Anantharaman-Leautaud-2014} for a detailed review on  stability of the damped wave equations. 

\begin{lemma}[Exponential stabilization of wave equations]\label{lem:lwedecay}
For any $T\geq 2\pi$ there exists some $J_T>0$ effectively computable such that the solution of the wave equation 
\begin{equation*}
    \begin{cases}
    -y_{tt}+ y_{xx}= a(x) y_t, \\
    y(0, x)= y_0, \; y_t(0, x)= y_1,
    \end{cases}
\end{equation*}
satisfies
\begin{equation*}
    J_T E_1(0)\leq 2\int_0^T \int_{\S} a(x) (y_t)^2(t, x)\, dx dt 
\end{equation*}
where 
\begin{equation*}
    E_1(t):= \int_{\S} (y_x(t, x))^2+ (y_t(t, x))^2 dx.
\end{equation*}
\end{lemma}

Return to the proof of Proposition~\ref{prop:3}. 
As we know that the damped wave maps equation 
\begin{equation*}
    \begin{cases}
   \Box \phi= \left(|\phi_t|^2- |\phi_x|^2\right)\phi+ a(x) \phi_{t}, \\
   (\phi, \phi_t)(0, x)= (g_0, g_1)(x),
    \end{cases}
\end{equation*}
 admits a unique solution $\phi$. Now let us consider the solution of the linearized damped wave maps equation 
 \begin{equation*}
    \begin{cases}
   \Box \phi_1= a(x) \phi_{1t}, \\
   (\phi_1, \phi_{1t})(0, x)= (g_0, g_1)(x).
    \end{cases}
\end{equation*}
Thanks to Lemma~\ref{lem:lwedecay}, the unique  solution of the linearized wave maps equation  also decays exponentially: in particular, by choosing $T= 16\pi$ the unique  solution $\phi_1$  verifies
\begin{equation*}
    J_{16\pi} E(0)\leq 2\int_0^{16\pi} \int_{\S} a(x) |\phi_{1t}|^2(t, x)\, dx dt. 
\end{equation*}
Thus
\begin{equation}\label{ineq:1}
    \int_{\S} |\phi_{1x}(16\pi, x)|^2+ |\phi_{1t}(16\pi, x)|^2 dx\leq \left(1- J_{16\pi}\right) E(0).
\end{equation}

By considering $\phi_2:= \phi- \phi_1$ we get 
\begin{equation*}
    \begin{cases}
   \Box \phi_2= a(x) \phi_{2t}+ \left(|\phi_t|^2- |\phi_x|^2\right)\phi, \\
   (\phi_2, \phi_{2t})(0, x)= (0, 0).
    \end{cases}
\end{equation*}
Therefore, for any $t\in [0, 16\pi]$, 
\begin{align*}
     \int_{\S} |\phi_{2x}(t, x)|^2+ |\phi_{2t}(t, x)|^2 dx &= -2\int_0^t \int_{\S} \left(|\phi_t|^2- |\phi_x|^2\right)\phi\cdot \phi_{2t}+ a(x)|\phi_{2t}|^2\, dx dt \\
     &\leq 2\int_0^t \int_{\S} \left(|\phi_t|^2- |\phi_x|^2\right)|\phi\cdot \phi_{2t}|\, dx dt \\
    &\leq  2 \|\phi_t\|_{L^{\infty}_t L^2_x(A_t)} \|\phi_t\|_{L^2_t L^{\infty}_x(A_t)} \|\phi_{2t}\|_{L^{2}_t L^2_x(A_t)} \\
    &\;\;\;\;\; + 2 \|\phi_x\|_{L^{\infty}_t L^2_x(A_t)} \|\phi_x\|_{L^2_t L^{\infty}_x(A_t)} \|\phi_{2t}\|_{L^{2}_t L^2_x(A_t)} \\
    &\lesssim E(0) \|\phi_{2t}\|_{L^{\infty}_t L^2_x(A_{16\pi})}
\end{align*}
where $A_t:=\{(s, y): s\in [0, t], y\in \S\}. $ This implies that 
\begin{equation}\label{ineq:2}
     \int_{\S} |\phi_{2x}(t, x)|^2+ |\phi_{2t}(t, x)|^2 dx \leq R_2 (E(0))^2, \; \forall t\in [0, 16\pi].
\end{equation}
Combine estimates \eqref{ineq:1} and \eqref{ineq:2},
\begin{align*}
    E(16\pi)&\leq (1+ J_{16\pi})  \int_{\S} |\phi_{1x}(16\pi, x)|^2+|\phi_{1t}(16\pi, x)|^2 dx\\
    &\;\;\;\;\; + (1+ J_{16\pi}^{-1})    \int_{\S} |\phi_{2x}(16\pi, x)|^2+ |\phi_{2t}(16\pi, x)|^2 dx \\
    &\leq (1- J_{16\pi}^2) E(0)+ (1+ J_{16\pi}^{-1}) R_2 (E(0))^2.
\end{align*}

By choosing $\mu_0$ in such fashion that $2(1+ J_{16\pi}^{-1}) R_2 \mu_0\leq J_{16\pi}^2$,  when $E(0)\leq \mu_0$ there is 
\begin{equation*}
     E(16\pi)\leq (1- J_{16\pi}^2/2) E(0),
\end{equation*}
which is equivalent to 
\begin{equation*}
    J_{16\pi}^2 E(0)\leq 4  \int_0^{16\pi} \int_{\S} a(x) |\phi_{t}|^2(t, x)\, dx dt. 
\end{equation*}
This ends the proof of Proposition~\ref{prop:3}.

\section{Semi-global exact controllability of wave maps}\label{sec:controlstep}

The semi-global controllability is proved  by two steps: we first stabilize the semi-global solution to some state with small energy with the help of damping control, then we  prove  local controllability  of  the controlled wave maps.  In this section we focus on the second step.  More precisely, in Section~\ref{subsec:2waveeque} we recall some controllability results on the related wave equations; then, we make some preparation on the controlled wave maps equation and present our strategy to  local exact controllability in Section~\ref{sec:strlec}; next, Section~\ref{subsec:2prooflocal} is devoted to the proof of local controllability using an iterative construction.

\subsection{Known controllability results on wave equations}\label{subsec:2waveeque}
We first recall the exact controllability result of wave equations in $\S$, which will be used later on to leading to the  local controllability of wave maps. 
\begin{lemma}[Controllability of wave equations]\label{lem:lwcontrol}
For any $T\geq 2\pi$ there exists some $\tilde J_T>0$ effectively computable such that, for any $(y_0, y_1)\in H^1_x\times L^2_x$ and any $(\tilde y_0, \tilde y_1)\in H^1_x\times L^2_x$ there exists some control  $f(t, x)\in L^2_{t, x}$ satisfying 
\begin{equation*}
    \|f\|_{L^{\infty}_t(0, T; L^2_x(\S))}\leq \tilde J_T \left(\|y_0\|_{H^1_x}+ \|y_1\|_{L^2_x}+ \|\tilde y_0\|_{H^1_x}+ \|\tilde y_1\|_{L^2_x}\right),
\end{equation*}
such that the unique solution of 
\begin{equation*}
 \Box y= {\bf 1}_{\omega} f, \; y[0]=(y_0, y_1),
\end{equation*}
verifies $y[T]= (\tilde y_0, \tilde y_1)$.
\end{lemma}
 Usually in the terminology of HUM the desired control candidate is chosen from the  $L^2(0, T; U)$ space (which also determines  the optimal control choice in this space). In the context of wave equations this control term can be expressed by the so-called HUM operator: let us define operators $\mathcal{R}$ and $\mathcal{S}$ as follows, 
\begin{align*}
    \mathcal{R}: L^2([0, T]\times \omega) &\rightarrow H^1\times L^2(\S) \\
    f&\mapsto y[0],
\end{align*}
where 
\begin{equation*}
    \Box y= {\bf 1}_{\omega} f, \; y[T]= (0, 0),
\end{equation*}
and 
\begin{align*}
    \mathcal{S}:  L^2\times H^{-1}(\S) &\rightarrow  L^2([0, T]\times \omega) \\
   (z_0, z_1) &\mapsto {\bf 1}_{\omega} z,
\end{align*}
where 
\begin{equation*}
    \Box z= 0, \; z[0]= (z_0, z_1).
\end{equation*}
Since $\mathcal{T}:= \mathcal{R} \circ \mathcal{S}$ is an isomorphism from $L^2\times H^{-1}(\S)$ to $H^1\times L^2(\S)$, for any $y[0]\in H^1\times L^2(\S)$ the optimal control that steers the state from $y[0]$ to $(0, 0)$ is given  by 
\begin{equation*}
    f:= \mathcal{S}\circ \mathcal{T}^{-1} (y[0]),
\end{equation*}
which also belongs to $L^{\infty}(0, T; L^2(\S))$.  One can refer to \cite[Section 1]{Dehman-Lebeau-2009}, \cite[Section 3.1]{Laurent-2011} for such  duality arguments  and more details on the  controllability of wave equations. Different from multidimensional cases where propagation of singularity techniques are adapted for GCC arguments, {\color{black} the  explicit observability inequality can be obtained via direct calculation on the one dimensional free wave equation.  }

\subsection{Strategy of the  local exact controllability}\label{sec:strlec}
Thanks to the damping stabilization, we  focus on the proof of the low-energy controllability result. We shall call some data $u[0]= (u, u_t)(x)$ to be \textit{$\varepsilon$-concentrated around $p$} provided that 
\begin{equation}
    |u(x)- p|\leq \varepsilon, \; \forall x\in \S.
\end{equation}
In particular, we notice that every data $u[0]= (u(x), u_t(x)): \S\rightarrow \mathbb{S}^k\times T\mathbb{S}^k$ and  $p\in \mathbb{S}^k$ satisfying \begin{equation*}
    \|u[0]- (p, 0)\|_{H^1_x\times L^2_x}\leq \varepsilon
\end{equation*}
is $3 \varepsilon$-concentrated around $p$, which, thanks to the geometry of the sphere, further implies that 
\begin{gather*}
   \left| \left(u_x(x)\right)^{p}\right|= |\langle u_x, p\rangle p|= |\langle u_x, p- u\rangle|\leq 3\varepsilon |u_x(x)|,\\
    \left| \left(u_t(x)\right)^{p}\right|\leq 3\varepsilon |u_t(x)| \; \textrm{ and } \;  \left| \left(u(x)- p\right)^{p}\right|\leq \frac{3\varepsilon}{2} \left|u(x)- p\right|.
\end{gather*}
Hence
\begin{equation*}
      \|\left(u[0]- (p, 0)\right)^p\|_{H^1_x\times L^2_x}\leq 3\varepsilon  \|u[0]- (p, 0)\|_{H^1_x\times L^2_x}
\end{equation*}

Recall that by $f^{p^{\perp}}$ we refer to the orthogonal projection of $f$ onto the plane $p^{\perp}$.
Let us start by  presenting the following well-posedness results:
\begin{lemma}[The inhomogeneous wave maps equation]\label{lem:inhwm}
Let $T= 2\pi$. For any initial sate  $\phi[0]: \S\rightarrow \mathbb{S}^k\times T\mathbb{S}^k$ in $H^1_x\times L^2_x$ and any source term $f(t, x): \S\rightarrow \mathbb{R}^{k+1}$ in $L^2_t L^2_x([0, T]\times \S)$ the inhomogeneous wave maps equation 
\begin{equation}\label{eq:inhomowavemap}
    \Box \phi= \left(|\phi_{t}|^2- |\phi_{x}|^2\right) \phi+ \mathbf{1}_{\omega} f^{\phi^{\perp}}
\end{equation}
admits a unique solution $\phi[t]$. Moreover, there exists some effectively computable constant $C_w>0$ such that  this unique solutions verifies the energy estimates
\begin{gather*}
     \|\phi[t]\|_{\dot{H}^1_x\times L^2_x}\leq \|\phi[0]\|_{\dot{H}^1_x\times L^2_x}+ \|f\|_{L^1_t L^2_x([0, T]\times \S)}, \; \forall t\in [0, T],\\
     \|(\phi_x, \phi_t)\|_{L^2_t\times L^{\infty}_x([0, T]\times \S)}\leq C_w \left(\|\phi[0]\|_{\dot{H}^1_x\times L^2_x}+ \|f\|_{L^1_t L^2_x([0, T]\times \S)}+ \|f\|_{L^2_t L^1_x([0, T]\times \S)}\right),
\end{gather*}
as well as  
\begin{gather*}
     (\phi, \phi_t)(t, x)\in \mathbb{S}^k\times T\mathbb{S}^k, \\
        |\phi(t, x)- p|\leq C_w \left(\|\phi[0]- (p, 0)\|_{H^1_x\times L^2_x}+ \|f\|_{L^1_t L^2_x([0, T]\times \S)}\right), \forall p\in \S,
\end{gather*}
for any $ (t, x)\in [0, T]\times \S$.
\end{lemma}
The energy estimates are similar to those of the damped wave maps (see Section~\ref{sec-basices} in particular the proof of \eqref{es:e3} for details). The first inequality on $L^{\infty}_t L^2_x$-norm of $(\phi_x, \phi_t)$ comes from  time derivation of the energy; then we extend $f(t)$ by 0 on $[0, T]^{c}$, which allows us to find some $x_0\in \S$ such that $\|(\phi_x, \phi_t)(t, x_0)\|_{L^2_t(-2\pi, T+ 2\pi)}$ is linearly bounded by $\sqrt{E(0)}$ and $\|f\|_{L^1_t L^2_x}$, thus 
\begin{equation*}
    \|(\phi_x, \phi_t)(t, x_0+y)\|_{L^2_t(-2\pi+y, T+ 2\pi-y)}- \|(\phi_x, \phi_t)(t, x_0)\|_{L^2_t(-2\pi, T+ 2\pi)}, \; \forall y\in [0, 2\pi],
\end{equation*}
is controlled by
$\|f\|_{L^1_x L^2_t}$.
Finally, we  comment on the  last estimate concerning $\varepsilon$-concentration: it is a direct consequence of  
\begin{gather*}
    \left|\int_{\S} \phi(t, x)- \phi(0, x)\, dx\right|=   \left|\int_{\S} \int_0^t \phi_t(s, x) ds dx\right|  \lesssim  \int_0^t\|\phi_t(s, x)\|_{L^2_x} ds, \\
     \left|\int_{\S} \phi(0, x)- p\, dx\right|\lesssim \|\phi(0, x)- p\|_{L^2_x},
\end{gather*}
and 
\begin{equation*}
    |\phi(t, x)- \phi(t, y)|\lesssim \|\phi_x(t, x)\|_{L^2_x}.
\end{equation*}

\begin{lemma}[The controlled wave equation]\label{lem:conlwm}
Let $T= 2\pi$.
For any source term $f(t, x): \S\rightarrow \mathbb{R}^{k+1}$ in $L^2_t L^2_x([0, T]\times \S)$, the unique solution of  the inhomogeneous wave equation 
\begin{equation}\label{eq:controlledlwe}
    \Box \phi=  f, \; \phi[0]= (0, 0),
\end{equation}
 verifies 
\begin{gather*}
     \|\phi[t]\|_{\dot{H}^1_x\times L^2_x}\leq \|f\|_{L^1_t L^2_x([0, T]\times \S)}, \; \forall t\in [0, T],\\
     \|(\phi_x, \phi_t)\|_{L^2_t\times L^{\infty}_x([0, T]\times \S)}\leq C_w \left(\|f\|_{L^1_t L^2_x([0, T]\times \S)}+ \|f\|_{L^2_t L^1_x([0, T]\times \S)}\right), \\
     \|\phi\|_{L^{\infty}_{t, x}([0, T]\times \S)}\leq C_w \left( \|f\|_{L^1_t L^2_x([0, T]\times \S)}\right).
\end{gather*}

There exists some effectively computable $G_T>0$ such that,  for any target state $(u_1, v_1)(x)\in H^1_x\times L^2_x(\S)$ one can find some explicit control function $\mathbf{1}_{\omega} f(t, x)$ satisfying 
\begin{equation}
    \|f\|_{L^{\infty}_t L^2_x([0, T]\times \S)}\leq G_T \|(u_1, v_1)\|_{H^1_x\times L^2_x},
\end{equation}
such that the unique solution of \eqref{eq:controlledlwe} with control $\mathbf{1}_{\omega} f$  verifies $\phi[T]= (u_1, v_1)$.
\end{lemma}

Now we are in position to state the local null controllability of the controlled wave maps equation: 
\begin{theorem}\label{thm:localcontrolwm}
Let $T= 2\pi$. There exist some effectively computable $\tilde \varepsilon>0$ such that for any initial state $u[0]= (a, b): \S\rightarrow \mathbb{S}^k\times T\mathbb{S}^k$ verifying 
\begin{equation*}
       \|(a, b)- (p, 0)\|_{H^1_x\times L^2_x}= \varepsilon\leq \tilde \varepsilon,
\end{equation*}
for some $p\in \mathbb{S}^k$,   we are able to construct a control $f$ satisfying 
\begin{equation*}
     \|f\|_{L^{\infty}_t L^2_x([0, T]\times \S)}\leq 50 G_T \varepsilon,
\end{equation*}
such that the unique solution of the inhomogeneous wave maps equation 
\begin{gather*}
     \Box \phi= \left(|\phi_{t}|^2- |\phi_{ x}|^2\right) \phi+ \mathbf{1}_{\omega} f^{\phi^{\perp}}, \; \phi[0]= (a, b),
\end{gather*}
 verifies 
 \begin{gather*}
 \phi[T]= (p, 0).
 \end{gather*}
\end{theorem}
Thanks to the time reversal of the controlled wave maps equation, Theorem~\ref{thm:localcontrolwm} implies  the local exact controllabilty of the controlled wave maps.  becasue $\mathbb{S}^k$ is compact and connected, there exists some integer $N$ such that  for any $p, q\in \mathbb{S}^k$ we are able to select a sequence  $\{p_i\}_{i=0}^N\subset \mathbb{S}^k$ satisfying $(p_0, p_N)= (p, q)$ in such fashion that 
\begin{equation*}
    |p_i- p_{i+1}|\leq \tilde \varepsilon. \; \forall i\in \{0, 1,..., N-1\}.
\end{equation*}
Further notice that we can construct a control to move the state from $(p_k, 0)$to $(p_{k+1}, 0)$ according to Theorem \ref{thm:localcontrolwm},   we arrive at the following detailed version of Theorem~\ref{thm:smallcontrolwm} concerning low-energy exact controllability of wave maps.
\begin{cor}\label{thm:localexactcontrolwm}
Let $T= 2\pi (N+ 1)$.  For any states $u[0], u[T]: \S\rightarrow \mathbb{S}^k\times T\mathbb{S}^k$ in $H^1_x\times L^2_x$ verifying 
\begin{equation*}
    \|u[0]\|_{\dot{H}^1_x\times L^2_x}, \|u[T]\|_{\dot{H}^1_x\times L^2_x}\leq \frac{(\tilde \varepsilon)^2}{100},
\end{equation*}
we are able to construct a control $f(t, x)$ satisfying 
\begin{equation*}
     \|f\|_{L^{\infty}_t L^2_x([0, T]\times \S)}\leq 5000G_{2\pi} \left(\|u[0]\|_{\dot{H}^1_x\times L^2_x}+ \|u[T]\|_{\dot{H}^1_x\times L^2_x}\right) ,
\end{equation*}
such that the unique solution of the inhomogeneous wave maps equation 
\begin{gather*}
     \Box \phi= \left(|\phi_{t}|^2- |\phi_{ x}|^2\right) \phi+ \mathbf{1}_{\omega} f^{\phi^{\perp}}, \; \phi[0]= u[0],
\end{gather*}
 verifies 
 \begin{gather*}
 \phi[T]= u[T].
 \end{gather*}
\end{cor}

\subsection{Proof of the local null controllability}\label{subsec:2prooflocal}
Now we are in position to prove Theorem~\ref{thm:localcontrolwm} concerning local null controllability. 
We adapt an iteration scheme to construct the required control function: for every $k\in \mathbb{N}$ we construct a pair $(\phi_k, f_{k-1})$ satisfying
\begin{equation}\label{eq:iteform}
    \Box \phi_k= \left(|\phi_{k, t}|^2- |\phi_{k, x}|^2\right) \phi_{k}+ \mathbf{1}_{\omega} f_{k-1}^{\phi^{\perp}_k}, \; \phi_k[0]= (a, b),
\end{equation}
such that $\phi_k[T]$ converges to $(p, 0)$ and that their limit, $(\tilde \phi, \tilde f)$,  is the desired solution:
\begin{gather*}
     \Box \tilde \phi= \left(|\tilde \phi_{ t}|^2- |\tilde \phi_{x}|^2\right) \tilde \phi+ \mathbf{1}_{\omega} \tilde f^{\tilde \phi^{\perp}}, \\
     \tilde \phi[0]= (a, b) \textrm{ and } \tilde \phi[T]= (p, 0).
\end{gather*}

\noindent {\bf Step 0.} In the zeroth iterate, namely $k=0$, we select  $f_{-1}= 0$, thus the system becomes a simple wave maps equation:
\begin{equation*}
    \Box \phi_0= \left(|\phi_{0, t}|^2- |\phi_{0, x}|^2\right) \phi_{0}, \; \phi_{0}[0]= (a, b).
\end{equation*}
Since 
\begin{equation}
    \|(a, b)- (p, 0)\|_{H^1_x\times L^2_x}= \varepsilon\leq \tilde \varepsilon,
\end{equation}
by conservation of the energy 
\begin{equation*}
    \|\phi_0[t]\|_{\dot H^1_x\times L^2_x}\leq \varepsilon, \; \forall t\in [0, T].
\end{equation*}
Furthermore, since
\begin{equation*}
    \left|\int_{\S} \phi_0(t, x)- \phi_0(0, x)\, dx\right|= \left|\int_0^t \int_{\S} \phi_{0, t}(s, x)\, dx dt\right|\leq 6\pi \varepsilon, \; \forall t\in [0, T],
\end{equation*}
we know that for any  $t\in [0, T]$, 
\begin{gather}
   \textrm{ $\phi_0[t]$ is $9\varepsilon$-concentrated around $p$},\\
     \|\phi_0[t]- (p, 0)\|_{H^1_x\times L^2_x}\leq 25\varepsilon, \\
      \|(\phi_{0, x}, \phi_{0, t})\|_{L^2_t\times L^{\infty}_x([0, T]\times \S)}\leq C_w \varepsilon.
\end{gather}
According to the geometry of  sphere we have the important observation that
\begin{equation}
     \|\left(\phi_0[t]- (p, 0)\right)^p\|_{H^1_x\times L^2_x}\leq 9\varepsilon\|\phi_0[t]- (p, 0)\|_{H^1_x\times L^2_x}.
\end{equation}
This finishes the zeroth iterate with error estimate
\begin{equation}
     \|\phi_0[T]- (p, 0)\|_{H^1_x\times L^2_x}\leq 25\varepsilon.
\end{equation}

\noindent {\bf Step 1.} We start the first iterate by fixing the error as much as possible  using some well chosen control $f_0(t, x)$, then construct the related function $\phi_1$.  Thanks to Lemma~\ref{lem:conlwm} concerning the exact controllability of  wave equations, we are able to select $f_0$ in such fashion that the solution of 
\begin{equation*}
    \Box \tilde \varphi_0= \mathbf{1}_{\omega} f_0, \;  \tilde \varphi_0[0]= (0, 0),
\end{equation*}
satisfies
\begin{equation*}
    \tilde \varphi_0[T]= -\phi_0[T]+ (p, 0), 
\end{equation*}
with
\begin{equation}
    \|f_0\|_{L^{\infty}_t L^2_x([0, T]\times \S)}\leq G_T \|\phi_0[T]- (p, 0)\|_{H^1_x\times L^2_x}\leq 25 G_T \varepsilon.
\end{equation}

Then we have that the slightly modified solution of 
\begin{equation*}
    \Box  \varphi_0= \mathbf{1}_{\omega} f_0^{p^{\perp}}, \;   \varphi_0[0]= (0, 0),
\end{equation*}
satisfies
\begin{gather*}
    \varphi_0[T]= -\phi_0[T]+ (p, 0) + \left(\phi_0[T]- (p, 0)\right)^p.
\end{gather*}
We also know from Lemma~\ref{lem:conlwm} that
\begin{gather*}
     \|\varphi_0[t]\|_{\dot{H}^1_x\times L^2_x}\leq 2\pi G_T \|\phi_0[T]- (p, 0)\|_{H^1_x\times L^2_x}\leq  50\pi G_T \varepsilon, \; \forall t\in [0, T],\\
     \|(\varphi_{0, x}, \varphi_{0, t})\|_{L^2_t\times L^{\infty}_x([0, T]\times \S)}\leq 8\pi G_T C_w \|\phi_0[T]- (p, 0)\|_{H^1_x\times L^2_x} \leq 200\pi G_T C_w \varepsilon, \\
     \|\varphi_0\|_{L^{\infty}_{t, x}([0, T]\times \S)}\leq 4\pi G_T C_w \|\phi_0[T]- (p, 0)\|_{H^1_x\times L^2_x} \leq 100\pi G_T C_w \varepsilon.
\end{gather*}

Notice that $\tilde \phi_1:= \phi_0+ \varphi_0$ verifies
\begin{gather*}
   \|\tilde \phi_1[T]- (p, 0)\|_{H^1_x\times L^2_x}= \|\left(\phi_0[T]- (p, 0)\right)^p\|_{H^1_x\times L^2_x}\leq 9 \varepsilon \|\phi_0[T]- (p, 0)\|_{H^1_x\times L^2_x},
\end{gather*}
 as well as  
\begin{gather*}
     \textrm{ $\tilde \phi_1[t]$ is $(9+ 100\pi G_T C_w)\varepsilon$-concentrated around $p$},\\
     \|\tilde \phi_1[t]- (p, 0)\|_{H^1_x\times L^2_x}\lesssim \varepsilon, \\
      \|(\tilde \phi_{1, x}, \tilde \phi_{1, t})\|_{L^2_t\times L^{\infty}_x([0, T]\times \S)}\lesssim \varepsilon.
\end{gather*}
Thus in principle we would like to set $\phi_1$ as $\tilde \phi_1$.
However, this still needs to be modified a bit to guarantee the geometric constraint, thus the form \eqref{eq:iteform}.  Observe from the construction of $\tilde \phi_1$ that 
\begin{equation}
    \Box \tilde \phi_1= \left(|\tilde \phi_{1, t}|^2- |\tilde \phi_{1, x}|^2\right) \tilde \phi_{1}+ \mathbf{1}_{\omega} f_{0}^{\tilde \phi^{\perp}_1}+ e_0, \; \tilde \phi_1[0]= (a, b)
\end{equation}
with 
\begin{align*}
    e_0(t, x)&= \phi_0\left(\langle 2\phi_{0, x}+ \varphi_{0, x}, \varphi_{0, x}\rangle- \langle 2\phi_{0, t}+ \varphi_{0, t}, \varphi_{0, t}\rangle\right) \\
    & \;\;\;\;\;\; - \varphi_0 \left(|\tilde \phi_{1, t}|^2- |\tilde \phi_{1, x}|^2 \right)+ \mathbf{1}_{\omega} \left(f_0^{p^{\perp}}- f_{0}^{\tilde \phi^{\perp}_1} \right).
\end{align*}
Combining the preceding estimates on $\phi_0$ and $\varphi_0$ we know that 
\begin{gather*}
    \| \varphi_0 \left(|\tilde \phi_{1, t}|^2- |\tilde \phi_{1, x}|^2 \right)\|_{L^2_{t, x}([0, T]\times \S)}\lesssim \varepsilon^2 \|\phi_0[T]- (p, 0)\|_{H^1_x\times L^2_x}, \\
     \| \phi_0\left(\langle 2\phi_{0, x}+ \varphi_{0, x}, \varphi_{0, x}\rangle- \langle 2\phi_{0, t}+ \varphi_{0, t}, \varphi_{0, t}\rangle\right)\|_{L^2_{t, x}([0, T]\times \S)}\lesssim \varepsilon \|\phi_0[T]- (p, 0)\|_{H^1_x\times L^2_x}.
\end{gather*}
Moreover, since $|\tilde \phi_1(t, x)- p|\leq C \varepsilon$, there is 
\begin{equation*}
    |g^{p^{\perp}}- g^{\tilde \phi^{\perp}_1}|\lesssim \varepsilon |g|, \; \forall g\in \mathbb{R}^{k+1},
\end{equation*}
which yields 
\begin{equation*}
    \|\mathbf{1}_{\omega} \left(f_0^{p^{\perp}}- f_{0}^{\tilde \phi^{\perp}_1} \right)\|_{L^2_{t, x}([0, T]\times \S)}\lesssim \varepsilon \|f_{0} \|_{L^2_{t, x}([0, T]\times \S)}\lesssim \varepsilon \|\phi_0[T]- (p, 0)\|_{H^1_x\times L^2_x}.
\end{equation*}
Thus
\begin{equation}
    \|e_0\|_{L^2_{t, x}([0, T]\times \S)}\lesssim \varepsilon \|\phi_0[T]- (p, 0)\|_{H^1_x\times L^2_x}.
\end{equation}

In order to eliminate the error  $e_0$ we try to find some suitable correction term $w_0$ such that $\phi_1:=\tilde \phi_1+ w_0= \phi_0+ \varphi_0+ w_0$ satisfies 
\begin{equation}\label{eq:phi1wm}
    \Box  \phi_1= \left(| \phi_{1, t}|^2- | \phi_{1, x}|^2\right)  \phi_{1}+ \mathbf{1}_{\omega} f_{0}^{\ \phi^{\perp}_1}, \;  \phi_1[0]= (a, b).
\end{equation}
According to Lemma~\ref{lem:inhwm} the Cauchy problem \eqref{eq:phi1wm} admits a unique solution satisfying 
\begin{gather*}
     |\phi_1(t, x)- p|\lesssim \varepsilon, \; \forall (t, x)\in [0, T]\times \S,  \\
     \|(\phi_{1, x}, \phi_{1, t})\|_{L^2_t\times L^{\infty}_x\cap L^{\infty}_t\times L^{2}_x([0, T]\times \S)}\lesssim \varepsilon.
\end{gather*}

In order to simplify the notations from now on we shall denote the $\mathcal{W}-$norm of a function $\phi[t]$ by
\begin{equation}
    \|\phi\|_{\mathcal{W}}:= \|(\phi_{x}, \phi_{t})\|_{L^2_t\times L^{\infty}_x\cap L^{\infty}_t\times L^{2}_x([0, T]\times \S)}.
\end{equation}

Next, we investigate the  correction term $w_0$ satisfying
\begin{equation*}
       \Box  w_0= g_0, \;  w_0[0]= (0, 0),
\end{equation*}
with 
\begin{equation*}
    g_0= \left(| \phi_{1, t}|^2- | \phi_{1, x}|^2\right)  \phi_{1}- \left(|\tilde \phi_{1, t}|^2- |\tilde \phi_{1, x}|^2\right) \tilde \phi_{1}+ \mathbf{1}_{\omega} f_{0}^{\ \phi^{\perp}_1}- \mathbf{1}_{\omega} f_{0}^{\tilde \phi^{\perp}_1}- e_0.
\end{equation*}
Thanks to Lemma~\ref{lem:conlwm}, there is 
\begin{equation*}
    \|w_0\|_{\mathcal{W}}+ \|w_0\|_{L^{\infty}_{t, x}([0, T]\times \S)}\lesssim \|g_0\|_{L^2_{t, x}([0, T]\times \S)}.
\end{equation*}
It suffices to estimate the value of $\|g_0\|_{L^2_{t, x}([0, T]\times \S)}$.  Because 
\begin{equation*}
    |\phi_1(t, x)- p|+  |\tilde \phi_1(t, x)- p|\lesssim \varepsilon, \; \forall (t, x)\in [0, T]\times \S,
\end{equation*}
by the geometry of the sphere we have 
\begin{equation*}
    |f_{0}^{\ \phi^{\perp}_1}- f_{0}^{\tilde \phi^{\perp}_1}|\lesssim \varepsilon |f_0|,
\end{equation*}
which further yields 
\begin{equation*}
    \|\mathbf{1}_{\omega} f_{0}^{\ \phi^{\perp}_1}- \mathbf{1}_{\omega} f_{0}^{\tilde \phi^{\perp}_1}\|_{L^2_{t, x}([0, T]\times \S)}\lesssim \varepsilon \|f_0\|_{L^2_{t, x}([0, T]\times \S)}\lesssim \varepsilon \|\phi_0[T]- (p, 0)\|_{H^1_x\times L^2_x}.
\end{equation*}
Observe that 
\begin{equation*}
    \left(| \phi_{1, t}|^2\right)  \phi_{1}- \left(|\tilde \phi_{1, t}|^2\right) \tilde \phi_{1}= \langle w_{0, t}, \phi_{1, t}+ \tilde \phi_{1, t}\rangle \phi_1+ |\tilde \phi_{1, t}|^2 w_0,
\end{equation*}
thus
\begin{align*}
    &\;\;\;\;\; 
    \| \left(| \phi_{1, t}|^2\right)  \phi_{1}- \left(|\tilde \phi_{1, t}|^2\right) \tilde \phi_{1}\|_{L^2_{t, x}} \\
    &\lesssim \|w_{0, t}\|_{L^{\infty}_t L^2_x} \|\phi_{1, t}+ \tilde \phi_{1, t}\|_{L^2_t L^{\infty}_x}+  \|w_0\|_{L^{\infty}_{t, x}} \| \tilde \phi_{1, t}\|_{\mathcal{W}}^2 \\
    &\lesssim \varepsilon \left(\|w_0\|_{\mathcal{W}}+ \|w_0\|_{L^{\infty}_{t, x}} \right).
\end{align*}
Similar estimates also holds for
\begin{equation*}
    \| \left(| \phi_{1, x}|^2\right)  \phi_{1}- \left(|\tilde \phi_{1, x}|^2\right) \tilde \phi_{1}\|_{L^2_{t, x}}.
\end{equation*}
Combining the preceding estimates we arrive at 
\begin{equation*}
     \|w_0\|_{\mathcal{W}}+ \|w_0\|_{L^{\infty}_{t, x}([0, T]\times \S)}\lesssim \varepsilon \|\phi_0[T]- (p, 0)\|_{H^1_x\times L^2_x}+ \varepsilon \left(\|w_0\|_{\mathcal{W}}+ \|w_0\|_{L^{\infty}_{t, x}} \right).
\end{equation*}
Therefore, 
\begin{equation}
     \|w_0\|_{\mathcal{W}}+ \|w_0\|_{L^{\infty}_{t, x}([0, T]\times \S)}\lesssim \varepsilon \|\phi_0[T]- (p, 0)\|_{H^1_x\times L^2_x}
\end{equation}
provided $\tilde \varepsilon$ sufficiently small. In particular, this implies that 
\begin{equation}
    \|\phi_1[T]- (p, 0)\|_{H^1_x\times L^2_x}\lesssim \varepsilon \|\phi_0[T]- (p, 0)\|_{H^1_x\times L^2_x},
\end{equation}
thus we have gained an extra $\varepsilon$.

In conclusion we are able to find some effectively computable constant $\mathcal{C}>0$ such that 
\begin{gather*}
    \|\phi_1\|_{\mathcal{W}}\leq \mathcal{C} \varepsilon, \\
    |\phi_1(t, x)- p|\leq \mathcal{C} \varepsilon, \; \forall (t, x)\in [0, T]\times \S,\\
     \|f_0\|_{L^{\infty}_t L^2_x([0, T]\times \S)}\leq 25 G_T \varepsilon,\\
     \|\phi_1[T]- (p, 0)\|_{H^1_x\times L^2_x}\leq \mathcal{C} \varepsilon \|\phi_0[T]- (p, 0)\|_{H^1_x\times L^2_x}.
\end{gather*}

\noindent {\bf Step 2.} In the following we repeat the iteration procedure. Inspired by the preceding conclusion we expect to construct $\{(\phi_k, f_{k-1})\}_k$ in such fashion that 
\begin{gather*}
    \|\phi_k\|_{\mathcal{W}}\leq 2\mathcal{C} \varepsilon, \\
    |\phi_k(t, x)- p|\leq 2\mathcal{C} \varepsilon, \; \forall (t, x)\in [0, T]\times \S,\\
     \|f_{k-1}\|_{L^{\infty}_t L^2_x([0, T]\times \S)}\leq 50 G_T \varepsilon,
\end{gather*}
satisfying 
\begin{equation*}
     \|\phi_k[T]- (p, 0)\|_{H^1_x\times L^2_x}\leq \mathcal{C} \varepsilon \|\phi_{k- 1}[T]- (p, 0)\|_{H^1_x\times L^2_x}.
\end{equation*}
This motivates us to prove the following lemma at first:

\begin{lemma}\label{lem:contkeyiteration}
Let $T= 2\pi$. There exist some effectively computable $\tilde \varepsilon>0$ and $\mathcal{C}_1>0$ such that for any initial state $u[0]= (a, b): \S\rightarrow \mathbb{S}^k\times T\mathbb{S}^k$ verifying 
\begin{equation*}
       \|(a, b)- (p, 0)\|_{H^1_x\times L^2_x}= \varepsilon\leq \tilde \varepsilon, 
\end{equation*}
if some pair $(\phi_k, f_{k-1})$ satisfies
\begin{gather*}
     \Box \phi_k= \left(|\phi_{k, t}|^2- |\phi_{k, x}|^2\right) \phi_{k}+ \mathbf{1}_{\omega} f_{k-1}^{\phi^{\perp}_k}, \; \phi_k[0]= (a, b),\\
         \|\phi_k\|_{\mathcal{W}}\leq 2\mathcal{C} \varepsilon, \\
    |\phi_k(t, x)- p|\leq 2\mathcal{C} \varepsilon, \; \forall (t, x)\in [0, T]\times \S,\\
     \|f_{k-1}\|_{L^{\infty}_t L^2_x([0, T]\times \S)}\leq 50 G_T \varepsilon,
\end{gather*}
then we are able to construct another pair $(\phi_{k+1}, f_{k})$ such that 
\begin{gather*}
 \Box \phi_{k+1}= \left(|\phi_{k+1, t}|^2- |\phi_{k+1, x}|^2\right) \phi_{k+1}+ \mathbf{1}_{\omega} f_{k}^{\phi^{\perp}_{k+1}}, \; \phi_{k+1}[0]= (a, b),\\
  \| \phi_{k+1}[T]- (p, 0)\|_{H^1_x\times L^2_x}\leq \mathcal{C}_1 \varepsilon \|\phi_k[T]- (p, 0)\|_{H^1_x\times L^2_x},\\
    \|\phi_{k+1}- \phi_k\|_{\mathcal{W}}+    \|\phi_{k+1}- \phi_k\|_{L^{\infty}_{t, x}([0, T]\times \S)}\leq \mathcal{C}_1 \|\phi_k[T]- (p, 0)\|_{H^1_x\times L^2_x}, \\
    \|f_k- f_{k-1}\|_{L^{\infty}_t L^2_x([0, T]\times \S)}\leq G_T \|\phi_k[T]- (p, 0)\|_{H^1_x\times L^2_x}.
\end{gather*}
\end{lemma}
Let us quickly comment on the use of Lemma~\ref{lem:contkeyiteration} to continue the iteration procedure. 
Clearly Case  $k=1$ verifies the conditions of Lemma~\ref{lem:contkeyiteration} according to Step 1.  By reducing the value of $\tilde \varepsilon$ if necessary a standard mathematical induction argument yields, 
\begin{gather*}
     \Box \phi_k= \left(|\phi_{k, t}|^2- |\phi_{k, x}|^2\right) \phi_{k}+ \mathbf{1}_{\omega} f_{k-1}^{\phi^{\perp}_k}, \; \phi_k[0]= (a, b),\\
         \|\phi_k\|_{\mathcal{W}}\leq (2- 2^{1-k})\mathcal{C} \varepsilon, \\
    |\phi_k(t, x)- p|\leq  (2- 2^{1-k})\mathcal{C} \varepsilon, \; \forall (t, x)\in [0, T]\times \S,\\
     \|f_{k-1}\|_{L^{\infty}_t L^2_x([0, T]\times \S)}\leq 25 (2- 2^{1-k}) G_T \varepsilon, \\
      \| \phi_{k}[T]- (p, 0)\|_{H^1_x\times L^2_x}\leq 2^{-k} \varepsilon^{3/2},
\end{gather*}
for all $k\in \mathbb{N}^*$.

\begin{proof}[Proof of Lemma~\ref{lem:contkeyiteration}]
We mimic the construction in Step 1: first we shall find some $f_k= f_{k-1}+ h_{k}$  to correct the error as much as possible (with the choice of $h_k$ by using the controllability of the linear wave equation), then we try to solve the inhomogeneous wave maps equation with the source term $f_k$  hoping that the solution $\phi_{k+1}$  becomes closer to $(p, 0)$.

If coincidently $\phi_k[T]- (p, 0)= (0, 0)$, then we stop the procedure as $(\phi_k, f_{k-1})$ is the required pair. Otherwise, we solve the linear control problem and  select $h_k$ in such fashion that the solution of 
\begin{equation*}
    \Box \tilde \varphi_k= \mathbf{1}_{\omega} h_k, \;  \tilde \varphi_k[0]= (0, 0),
\end{equation*}
satisfies
\begin{equation*}
    \tilde \varphi_k[T]= -\phi_k[T]+ (p, 0), 
\end{equation*}
with
\begin{equation}
    \|h_k\|_{L^{\infty}_t L^2_x([0, T]\times \S)}\leq G_T \|\phi_k[T]- (p, 0)\|_{H^1_x\times L^2_x}.
\end{equation}

Again, thanks to the geometry of  sphere, we can define $\varphi_k$ as the solution of 
\begin{equation*}
    \Box  \varphi_k= \mathbf{1}_{\omega} h_k^{p^{\perp}}, \;   \varphi_k[0]= (0, 0)
\end{equation*}
satisfying, according to  Lemma~\ref{lem:conlwm},
\begin{gather*}
    \|\varphi_k\|_{\mathcal{W}}+   \|\varphi_k\|_{L^{\infty}_{t, x}([0, T]\times \S)}\lesssim  \|\phi_k[T]- (p, 0)\|_{H^1_x\times L^2_x}\lesssim \varepsilon.
\end{gather*}

By defining $\tilde \phi_{k+1}= \phi_k+ \varphi_k$ and $f_k= f_{k-1}+ h_k$ we get 
\begin{gather*}
    \|\tilde \phi_{k+1}[T]- (p, 0)\|_{H^1_x\times L^2_x}=  \|\left(\phi_k[T]- (p, 0)\right)^p\|_{H^1_x\times L^2_x}\lesssim \varepsilon \|\phi_k[T]- (p, 0)\|_{H^1_x\times L^2_x},
\end{gather*}
as well as 
\begin{gather*}
     \|\tilde \phi_{k+1}\|_{\mathcal{W}}\lesssim \varepsilon \; \textrm{ and } \; |\tilde \phi_{k+1}(t, x)- p|\lesssim \varepsilon, \; \forall (t, x)\in [0, T]\times \S.
\end{gather*}
The function $\tilde \phi_{k+1}$ verifies the equation 
\begin{equation}
    \Box \tilde \phi_{k+1}= \left(|\tilde \phi_{k+1, t}|^2- |\tilde \phi_{k+1, x}|^2\right) \tilde \phi_{k+1}+ \mathbf{1}_{\omega} f_{k}^{\tilde \phi^{\perp}_{k+1}}+ e_k, \; \tilde \phi_{k+ 1}[0]= (a, b)
\end{equation}
with 
\begin{align*}
    e_k&= \phi_k\left(\langle 2\phi_{k, x}+ \varphi_{k, x}, \varphi_{k, x}\rangle- \langle 2\phi_{k, t}+ \varphi_{k, t}, \varphi_{k, t}\rangle\right) \\
    & \;\;\;\;\;\; - \varphi_k \left(|\tilde \phi_{k+1, t}|^2- |\tilde \phi_{k+1, x}|^2 \right)+ \mathbf{1}_{\omega} \left(h_k^{p^{\perp}}+ f_{k-1}^{\phi_k^{\perp}}- f_{k}^{\tilde \phi^{\perp}_{k+1}} \right).
\end{align*}
Immediately we have 
\begin{gather*}
    \|\phi_k\left(\langle 2\phi_{k, x}+ \varphi_{k, x}, \varphi_{k, x}\rangle- \langle 2\phi_{k, t}+ \varphi_{k, t}, \varphi_{k, t}\rangle\right) \|_{L^2_{t, x}} \lesssim \varepsilon\|\phi_k[T]- (p, 0)\|_{H^1_x\times L^2_x},\\
     \| \varphi_k \left(|\tilde \phi_{k+1, t}|^2- |\tilde \phi_{k+1, x}|^2 \right)\|_{L^2_{t, x}} \lesssim \varepsilon^2\|\phi_k[T]- (p, 0)\|_{H^1_x\times L^2_x}
\end{gather*}
Furthermore, since 
\begin{gather*}
    |\tilde \phi_{k+1}- \phi_k|\lesssim \|\phi_k[T]- (p, 0)\|_{H^1_x\times L^2_x}, \\
     |\tilde \phi_{k+1}- p|\lesssim \varepsilon,
\end{gather*}
we obtain 
\begin{gather*}
    \|f_{k-1}^{\phi_k^{\perp}}- f_{k-1}^{\tilde \phi^{\perp}_{k+1}}\|_{L^2_{t, x}}\lesssim \|\phi_k[T]- (p, 0)\|_{H^1_x\times L^2_x} \|f_{k-1}\|_{L^2_{t, x}}\lesssim \varepsilon \|\phi_k[T]- (p, 0)\|_{H^1_x\times L^2_x},\\
      \|h_{k}^{p^{\perp}}- h_{k}^{\tilde \phi^{\perp}_{k+1}}\|_{L^2_{t, x}}\lesssim \varepsilon \|h_{k}\|_{L^2_{t, x}}\lesssim \varepsilon \|\phi_k[T]- (p, 0)\|_{H^1_x\times L^2_x}.
\end{gather*}
The preceding estimates lead to 
\begin{equation}
     \|e_k\|_{L^2_{t, x}([0, T]\times \S)}\lesssim \varepsilon \|\phi_k[T]- (p, 0)\|_{H^1_x\times L^2_x}.
\end{equation}

Finally we replace $\tilde \phi_{k+1}$ by $\phi_{k+1}= \tilde \phi_{k+1}+ w_k= \phi_k+ \varphi_k+ w_k$ with some suitable correction term $w_k$ satisfying 
\begin{equation}
    \Box  \phi_{k+1}= \left(| \phi_{k+1, t}|^2- | \phi_{k+1, x}|^2\right)  \phi_{k+1}+ \mathbf{1}_{\omega} f_{k}^{ \phi^{\perp}_{k+1}}, \;  \phi_{k+ 1}[0]= (a, b).
\end{equation}
According to Lemma~\ref{lem:inhwm} the preceding Cauchy problem admits a unique solution satisfying 
\begin{gather*}
     |\phi_{k+ 1}(t, x)- p|\lesssim \varepsilon, \; \forall (t, x)\in [0, T]\times \S,  \\
     \|\phi_{k+1}\|_{\mathcal{W}}\lesssim \varepsilon,
\end{gather*}
which implies that $|w_k|\lesssim \varepsilon$. In order to get better estimates on $w_k$ we investigate its equation as 
\begin{equation*}
       \Box  w_k= g_k, \;  w_k[0]= (0, 0),
\end{equation*}
with 
\begin{equation*}
    g_k= \left(| \phi_{k+1, t}|^2- | \phi_{k+1, x}|^2\right)  \phi_{k+1}- \left(|\tilde \phi_{k+1, t}|^2- |\tilde \phi_{k+1, x}|^2\right) \tilde \phi_{k+1}+ \mathbf{1}_{\omega} f_{k}^{\ \phi^{\perp}_{k+1}}- \mathbf{1}_{\omega} f_{k}^{\tilde \phi^{\perp}_{k+1}}- e_k.
\end{equation*}
Thanks to Lemma~\ref{lem:conlwm}, there is 
\begin{equation*}
    \|w_k\|_{\mathcal{W}}+ \|w_k\|_{L^{\infty}_{t, x}([0, T]\times \S)}\lesssim \|g_k\|_{L^2_{t, x}([0, T]\times \S)}.
\end{equation*}

We have 
\begin{equation*}
    \left(| \phi_{k+1, t}|^2\right)  \phi_{k+1}- \left(|\tilde \phi_{k+1, t}|^2\right) \tilde \phi_{k+1}= \langle w_{k, t}, \phi_{k+1, t}+ \tilde \phi_{k+1, t}\rangle \phi_{k+1}+ |\tilde \phi_{k+1, t}|^2 w_k,
\end{equation*}
thus
\begin{align*}
    \| \left(| \phi_{k+1, t}|^2\right)  \phi_{k+1}- \left(|\tilde \phi_{k+1, t}|^2\right) \tilde \phi_{k+1}\|_{L^2_{t, x}} \lesssim \varepsilon \left(\|w_k\|_{\mathcal{W}}+ \|w_k\|_{L^{\infty}_{t, x}} \right),
\end{align*}
and similarly
\begin{align*}
    \| \left(| \phi_{k+1, x}|^2\right)  \phi_{k+1}- \left(|\tilde \phi_{k+1, x}|^2\right) \tilde \phi_{k+1}\|_{L^2_{t, x}} \lesssim \varepsilon \left(\|w_k\|_{\mathcal{W}}+ \|w_k\|_{L^{\infty}_{t, x}} \right).
\end{align*}
Since 
\begin{equation*}
    |\tilde \phi_{k+1}- \phi_{k+1}|\leq \|w_k\|_{L^{\infty}_{t, x}}\ll 1,
\end{equation*}
we have 
\begin{equation*}
    \|f_{k}^{\ \phi^{\perp}_{k+1}}-  f_{k}^{\tilde \phi^{\perp}_{k+1}}\|_{L^2_{t, x}}\lesssim \|w_k\|_{L^{\infty}_{t, x}} \|f_{k}\|_{L^2_{t, x}}\lesssim \varepsilon\|w_k\|_{L^{\infty}_{t, x}}.
\end{equation*}
Combine the preceding estimates we arrive at 
\begin{equation*}
     \|w_k\|_{\mathcal{W}}+ \|w_k\|_{L^{\infty}_{t, x}([0, T]\times \S)}\lesssim  \varepsilon \left(\|w_k\|_{\mathcal{W}}+ \|w_k\|_{L^{\infty}_{t, x}} \right)+ \varepsilon \|\phi_k[T]- (p, 0)\|_{H^1_x\times L^2_x},
\end{equation*}
hence, for $\tilde \varepsilon$ sufficiently small 
\begin{equation}
     \|w_k\|_{\mathcal{W}}+ \|w_k\|_{L^{\infty}_{t, x}([0, T]\times \S)}\lesssim   \varepsilon \|\phi_k[T]- (p, 0)\|_{H^1_x\times L^2_x}.
\end{equation}

Therefore, there exists some effectively computable $\mathcal{C}_1>0$ such that 
\begin{gather*}
     \| \phi_{k+1}[T]- (p, 0)\|_{H^1_x\times L^2_x}\leq \mathcal{C}_1 \varepsilon \|\phi_k[T]- (p, 0)\|_{H^1_x\times L^2_x}, \\
      \|\phi_{k+1}- \phi_k\|_{\mathcal{W}}+    \|\phi_{k+1}- \phi_k\|_{L^{\infty}_{t, x}([0, T]\times \S)}\leq \mathcal{C}_1 \|\phi_k[T]- (p, 0)\|_{H^1_x\times L^2_x}, \\
    \|f_k- f_{k-1}\|_{L^{\infty}_t L^2_x([0, T]\times \S)}\leq G_T \|\phi_k[T]- (p, 0)\|_{H^1_x\times L^2_x}.
\end{gather*}
This ends the proof of Lemma~\ref{lem:contkeyiteration}.
\end{proof}

\noindent {\bf Step 3.}
By the construction of $\{(\phi_k, f_{k-1})\}_{k\in \mathbb{N}^*}$ we know that 
\begin{gather*}
    \{\phi_k[t]\}_{k\in \mathbb{N}^*} \textrm{ is a Cauchy sequence in } C^0_t([0, T]; H^1_x\times L^2_x(\S)),\\
     \{(\phi_{k, x}, \phi_{k, t})\}_{k\in \mathbb{N}^*} \textrm{ is a Cauchy sequence in }  L^2_t L^{\infty}_x([0, T]\times \S), \\
      \{\phi_k\}_{k\in \mathbb{N}^*} \textrm{ is a Cauchy sequence in } C^0_{t, x}([0, T]\times \S),\\
      \{f_{k-1}\}_{k\in \mathbb{N}^*} \textrm{ is a Cauchy sequence in }  C^0_t([0, T]; L^2_x(\S)).
\end{gather*}
Hence there exist a pair $(\phi, f)$ such that 
\begin{align*}
    (\phi_k, \phi_{k, t})&\rightarrow (\phi, \phi_t) \; \textrm{ in }  C^0_t([0, T]; H^1_x\times L^2_x(\S)), \\
     (\phi_{k, x}, \phi_{k, t})&\rightarrow (\phi_x, \phi_t) \; \textrm{ in }  L^2_t L^{\infty}_x([0, T]\times \S),\\
    \phi_k &\rightarrow \phi \textrm{ in } \; C^0_{t, x}([0, T]\times \S), \\
    f_{k-1}&\rightarrow f \textrm{ in } \; C^0_t([0, T]; L^2_x(\S)).
\end{align*}
By passing the limit of those equations 
satisfied by $\{(\phi_k, f_{k-1})\}_k$
we easily conclude the following equation on $\phi$:
\begin{gather*}
       \Box \phi= \left(|\phi_{t}|^2- |\phi_{x}|^2\right) \phi+ \mathbf{1}_{\omega} f^{\phi^{\perp}}, \\
       \phi[0]= (a, b) \; \textrm{ and } \;  \phi[T]= (0, 0).
\end{gather*}
\label{subsec:2globalcontrol}
\subsection{Semi-global exact controllability}
Let given $\nu>0$. According to the exponential stability of the damped wave maps equation Theorem~\ref{thm:main} (or the asymptotic stability guaranteed by Propositions~\ref{prop:1w}--\ref{prop:2}), there exists some effectively computable $T_1>0$  such that for any initial state $\phi[0]: \S\rightarrow \mathbb{S}^k\times T\mathbb{S}^k$  satisfying 
\begin{equation*}
    E(0)\leq 2\pi- \nu,
\end{equation*}
the unique solution of the damped wave maps equation verifies $E(T_1)\leq (\tilde \varepsilon)^2/100$. Then, thanks to Corollary~\ref{thm:localexactcontrolwm} we are able to find some effectively computable $T_2>0$ and  control $f\in L^{\infty}_t(T_1, T_1+ T_2; L^2_x(\S))$ such that the controlled wave maps equation  has final state $(1, 0)$. By the time reversal of the controlled wave maps equation, this implies that the wave maps equation is exact controllable in time period $2(T_1+ T_2)$ for states with energy below $2\pi- \nu$.

\bibliographystyle{plain}    
\bibliography{wavemapcontrol2}

\end{document}